\documentclass[11pt]{article}
\usepackage[curve]{xypic}
\usepackage{graphicx}
\usepackage{longtable}
\usepackage{rotating}
\usepackage{multirow}
\usepackage[active]{srcltx}
\usepackage[english]{babel}
\usepackage{amsmath,amsfonts,amssymb,amscd,color,epsfig,amsthm}
\usepackage{titletoc}
\usepackage{mathrsfs}
\usepackage{appendix}
\usepackage{enumitem}

\usepackage{hyperref}
\usepackage{upgreek}
\usepackage{tikz,tikz-cd}
\usepackage[normalem]{ulem}
\usepackage{changes}

\setaddedmarkup{\textcolor{blue}{#1}}
\setdeletedmarkup{\textcolor{blue}{\sout{#1}}}

\newcommand{\addedtwo}[1]{\textcolor{purple}{#1}}

\setenumerate[1]{itemsep=0pt,partopsep=0pt,parsep=\parskip,topsep=5pt}
\setitemize[1]{itemsep=0pt,partopsep=0pt,parsep=\parskip,topsep=5pt}
\setdescription{itemsep=0pt,partopsep=0pt,parsep=\parskip,topsep=5pt}
\usepackage{caption}

\numberwithin{equation}{section}

\newtheorem{theorem}{Theorem}[section]
\newtheorem{lemma}{Lemma}[section]
\newtheorem{proposition}{Proposition}[section]

\newtheorem{claim}{Claim}[section]

\newtheorem*{theoremA}{Theorem A}

\newtheorem*{theoremM}{McMullen's theorem}
\newtheorem*{proposition I}{Proposition I}
\newtheorem*{proposition II}{Proposition II}

\newcommand{\sm}{\setminus}

\renewcommand{\added}[1]{#1}
 \renewcommand{\addedtwo}[1]{#1}
\renewcommand{\deleted}[1]{}

\begin{document}
	\title{Rigidity of Simple Parabolic Maps}
	\author{Ning Gao
		\and Yan Gao\thanks{The second author is supported by the National Key R\&D Program of China Grant No. 2021YFA1003203, the NSFC Grant Nos. 12131016 and
			12322104, and the NSFGD Grant No. 2023A1515010058.}
		\and Wenjuan Peng\thanks{The third author is supported by the NSFC Grant Nos. 12671099, 12271115 and 12288201.}}
	\date{\today}
	\maketitle
	
\begin{abstract}
	We prove that two simple parabolic rational maps with conformally conjugate completely
	invariant parabolic Fatou domains are globally conformally conjugate. The proof combines
	quasiconformal rigidity with the absence of invariant line fields for simple parabolic maps.
	Together with our previous existence theorem, the rigidity result gives existence and uniqueness
	of simple parabolic models for infinitely-connected fixed parabolic Fatou domains. We further
	obtain a rational-map analogue of the Goldberg--Milnor perturbation conjecture.
\end{abstract}

\section{Introduction}

Let $f$ be a rational map of the Riemann sphere $\widehat{\mathbb C}$ of degree $d\geq 2$.
The Fatou set $F(f)$ is the set on which the iterates $\{f^n\}_{n\geq 0}$ form a normal
family locally, and the Julia set $J(f)$ is its complement.
Both sets are completely invariant. A connected component of $F(f)$, respectively $J(f)$,
is called a Fatou domain, respectively a Julia component. Sullivan \cite{Sullivan} proved that
rational maps have no wandering Fatou domains. Periodic Fatou domains are classified as
(super)attracting domains, parabolic domains, Siegel disks, and Herman rings; see
\cite{F,H,Milnor,S}.

The dynamics of Siegel disks and Herman rings are classically modeled by irrational rotations.
For fixed attracting and parabolic domains, we consider the stronger \textbf{global realization
problem}: finding a canonical rational map with a completely invariant Fatou domain on which
the dynamics is conformally conjugate to that of the original map. The simply-connected
attracting and parabolic cases are modeled by suitable finite Blaschke products, and the
infinitely-connected attracting case was treated by Cui and Peng \cite{CP}. Thus the remaining
case in this global rational-model problem is that of infinitely-connected fixed parabolic
domains; see \cite{GGP} and the references therein.

To address this remaining case, we introduced in \cite{GGP} the following class of rational maps.
A rational map is called a \textbf{simple parabolic map} if it has a completely invariant parabolic
domain and every non-singleton Julia component is a Jordan curve bounding an eventually
superattracting domain that contains at most one postcritical point. Moreover, every such
Julia component is a quasicircle whenever its forward orbit avoids the parabolic fixed point.

Now let $U$ be an infinitely-connected fixed parabolic domain of $f$. We seek a holomorphic
model of $(f,U)$ in the class of simple parabolic maps. More precisely, a simple parabolic map
$g$, with a completely invariant parabolic domain $V$, is called a \textbf{holomorphic model} of
$(f,U)$ if
\begin{enumerate}
	\item[(1)] $(f,U)$ and $(g,V)$ are conformally conjugate;\footnote{For invariant domains
		$D_i$ of rational maps $R_i$, we say that $(R_1,D_1)$ and $(R_2,D_2)$ are topologically,
		quasiconformally, or conformally conjugate if there exists, respectively, a homeomorphism,
		a quasiconformal map, or a conformal map $\phi:D_1\to D_2$ satisfying
		$\phi\circ R_1=R_2\circ\phi$ on $D_1$.}
	\item[(2)] whenever $\widetilde g$ is another simple parabolic map with a completely invariant
	parabolic domain $\widetilde V$ such that $(f,U)$ and $(\widetilde g,\widetilde V)$ are
	conformally conjugate, $g$ and $\widetilde g$ are conformally conjugate on
	$\widehat{\mathbb C}$.
\end{enumerate}

Our previous work \cite{GGP} proves that the first condition can always be satisfied:

\begin{theoremA}{\rm(\cite[Theorem 1.1]{GGP})}
	Let $f$ be a rational map with an infinitely-connected fixed parabolic Fatou domain $U$.
	Then there exists a simple parabolic map $g$ with a completely invariant parabolic Fatou
	domain $V$ such that $(f,U)$ and $(g,V)$ are conformally conjugate.
\end{theoremA}

Theorem A settles existence in the chosen model class, leaving uniqueness as the remaining
issue for simple parabolic models. The main theorem of the present paper proves this uniqueness
through rigidity.

\begin{theorem}[Conformal rigidity]\label{main theorem}
	Let $f$ and $g$ be two simple parabolic maps with completely invariant parabolic Fatou domains
	$U$ and $V$, respectively. If $(f,U)$ and $(g,V)$ are conformally conjugate, then $f$ and $g$
	are conformally conjugate on $\widehat{\mathbb C}$.
\end{theorem}

A simple parabolic map is therefore determined, up to global conformal conjugacy, by the
conformal conjugacy class of its dynamics on the completely invariant parabolic Fatou domain.
Together with Theorem A, this gives existence and uniqueness of simple parabolic models for
infinitely-connected fixed parabolic Fatou domains.

Theorem~\ref{main theorem} is a rigidity result with a local-to-global character: the assumed
conformal conjugacy is defined only on the completely invariant parabolic Fatou domains, whereas
the conclusion asserts that the rational maps themselves are conformally conjugate on the whole
sphere. Generally, rigidity in complex dynamics means that a weak form of conjugacy implies a
stronger one. It is a central theme in the subject. For example, the well-known MLC conjecture,
which asserts that the Mandelbrot set is locally connected, is equivalent to the combinatorial
rigidity conjecture for quadratic polynomials; see \cite{RS}.

The proof of Theorem~\ref{main theorem} has two main ingredients. The first is a quasiconformal
rigidity theorem.

\begin{theorem}[Quasiconformal rigidity]\label{qc conj}
	Suppose that $f$ and $g$ are two simple parabolic maps with completely invariant parabolic
	Fatou domains $U$ and $V$, respectively. If $(f,U)$ and $(g,V)$ are conformally conjugate,
	then $f$ and $g$ are quasiconformally conjugate on $\widehat{\mathbb C}$.
\end{theorem}

The second ingredient is the absence of invariant line fields, which is also of independent
interest. It is conjectured that every rational map,
	except for flexible Latt\`es maps, carries no invariant line fields on its
	Julia set. This conjecture is known to be stronger than the
	Hyperbolicity Conjecture \cite{MS}.

\begin{theorem}[Absence of invariant line fields]\label{line field}
	A simple parabolic map carries no invariant line fields on its Julia set.
\end{theorem}

Theorems~\ref{qc conj} and~\ref{line field} together imply Theorem~\ref{main theorem}. Indeed, Theorem~\ref{qc conj} provides a global quasiconformal conjugacy extending the conformal conjugacy on the parabolic domains. If this quasiconformal conjugacy were not conformal, its Beltrami coefficient would define an $f$-invariant line field on $J(f)$, contradicting Theorem~\ref{line field}.

Both quasiconformal rigidity and the absence of invariant line fields have been extensively
studied in complex dynamics. Quasiconformal rigidity is known in a number of settings, including
real polynomials \cite{KSS}, non-renormalizable polynomials and certain infinitely renormalizable
unicritical polynomials \cite{KvS09,Che10}, rational maps with Cantor Julia sets \cite{Zhai},
Newton maps \cite{DS22,RYZ23}, and certain polynomials with parabolic points \cite{CWY}.
The absence of invariant line fields has been established for non-recurrent and parabolic
Collet--Eckmann maps \cite{PU01}, weakly hyperbolic maps \cite{Hai}, $J$-rotational maps
\cite{Lim26}, rational maps with Cantor Julia sets \cite{YZ}, and maps satisfying suitable
summability conditions \cite{GS09,Mak05}.

For simple parabolic maps, however, the existing arguments do not apply directly. The proof
of Theorem~\ref{qc conj} proceeds in three stages. First, we construct a local quasiconformal
conjugacy around every non-singleton periodic Julia component. Second, by lifting these
conjugacies to \addedtwo{all non-singleton} components and combining them with the original conformal
conjugacy on the completely invariant parabolic Fatou domains, we obtain a global topological
conjugacy on $\widehat{\mathbb C}$. Finally, we prove that this topological conjugacy is
quasiconformal.

A specific obstruction arises in the first stage. When the parabolic fixed point belongs to a
non-singleton Julia component, boundaries of puzzle pieces of different depths may meet at the
parabolic point and at its iterated preimages. This prevents the usual quasiconformal extension
between nested puzzle pieces.

To overcome this obstruction, we construct suitable thickened puzzle neighborhoods around
non-singleton periodic Julia components. Near the parabolic fixed point, we combine the given
conformal conjugacy on the parabolic Fatou domain with a local quasiconformal extension supplied
by \cite[Lemma 3.4]{CT}. We then extend the map across the thickened puzzle neighborhood
and prove its quasiconformality by applying the criterion of Cao, Wang and Yin
\cite[Proposition A.1]{CWY}, an analogue of the extremal-length criterion of Kozlovski, Shen
and van Strien \cite[Lemma 12.1]{KSS}. This provides the local quasiconformal conjugacies
needed for the global construction.

For the last step, we decompose the Julia set according to critical recurrence and isolate the
persistently recurrent part on which the Kozlovski--Shen--van Strien (KSS) nest \cite{KSS}
is available. Uniform geometric estimates for these nests allow us to verify the quasiconformal
criterion and are also used in the proof of Theorem~\ref{line field}.
In the proof of Theorem~\ref{line field}, the key additional point is a critical-orbit separation that excludes extra critical points from the relevant pullbacks. Together with the KSS degree bounds, this provides the uniform geometric control required by Shen's line-field criterion \cite{Shen}.

Theorem~\ref{main theorem} also completes the perturbation picture developed in \cite{GGP}.
Goldberg and Milnor \cite{GM} conjectured that, for any polynomial with a parabolic cycle, the
immediate basin of that cycle can be converted into an attracting basin by a small perturbation,
while the dynamics of the perturbed polynomial on its Julia set are topologically conjugate to
those of the original polynomial on its Julia set. We obtain the following analogue for simple
parabolic maps.

\begin{theorem}[Goldberg--Milnor type perturbation]\label{perturbation}
	Every simple parabolic map can be approximated by rational maps with completely invariant
	attracting Fatou domains, and the dynamics on their Julia sets are topologically conjugate
	to those of the original map.
\end{theorem}

For simple parabolic maps with Cantor Julia sets, the rigidity and perturbation statements
are already known: quasiconformal rigidity follows from \cite{Zhai}, the absence of invariant
line fields from \cite{YZ}, and the perturbation statement from \cite[Theorem 1.3]{GGP}.
The present paper therefore concentrates on the case in which non-singleton Julia components
occur.

The paper is organized as follows. In Section 2 we introduce the puzzle structure for simple
parabolic maps. In Section 3 we construct local quasiconformal conjugacies around non-singleton
periodic Julia components, treating separately the cases in which the parabolic Julia component
is non-singleton or a singleton. In Section 4 we extend these local conjugacies to a global
topological conjugacy and prove Theorem~\ref{qc conj}. Section 5 is devoted to
Theorem~\ref{line field}. Finally, in Section 6 we deduce Theorems~\ref{main theorem}
and~\ref{perturbation}.

Throughout the paper, $\operatorname{length}$ and $\operatorname{mes}$ denote one-dimensional
Hausdorff measure and planar Lebesgue measure, respectively. We use the term component for
connected component. For a set $A\subset\mathbb C$ and a connected subset $E\subset A$,
$\operatorname{Comp}_E A$ denotes the component of $A$ containing $E$. For a Jordan curve \(\gamma \subset \mathbb{C}\), \(\operatorname{Int}(\gamma)\) denotes the bounded component of \(\mathbb{C} \setminus \gamma\).

\section{Puzzles for a simple parabolic map}
In this section, we introduce the construction and basic properties of puzzles for simple parabolic maps. 

Suppose that $f$ and $g$ are simple parabolic maps with completely invariant parabolic Fatou domains $U$ and $V$, respectively. Let $h$ be the conformal map from $U$ to $V$ satisfying $h \circ f = g \circ h$. Let $C(f)$ denote the set of critical points of $f$, and $P(f)$ denote the postcritical set, defined as
\[P(f)=\overline{\{f^n(c) \mid c \in C(f), \, n\geq1\}}.\] The sets $C(g)$ and $P(g)$ are defined analogously for the map $g$.

Let $x_f$ and $x_g$ be the parabolic fixed points of $f$ and $g$, respectively. Choose a point $b\in U$. After conjugating $f$ and $g$ separately by suitable M\"obius transformations and replacing $h$ accordingly, we may assume that $x_f=x_g=0$ and $b=h(b)=\infty$. We still denote the resulting maps and domains by $f,g,U,V$, and the corresponding conformal conjugacy from $U$ onto $V$ by $h$. Thus, $U$ and $V$ are the immediate parabolic basins of $0$ for $f$ and $g$, respectively, and $h(\infty)=\infty$.

Denote by $\mathcal{E}_f$ and $\mathcal{E}_g$ the collections of all components of $\partial U$ and $\partial V$, respectively. The map $f$ induces a surjective map $\sigma_f\colon \mathcal{E}_f \to \mathcal{E}_f$. For each $E \in \mathcal{E}_f$, its \textbf{filling} $\widehat{E}$ is defined to be the union of $E$ and all bounded components of $\mathbb{C}\sm E$. We define the degree of $\sigma_f$ on $E$ by
\[
\deg_E\sigma_f
:=
\sum_{z\in \widehat E}(\deg_z f-1)+1.
\]
If $\deg_E\sigma_f>1$, we say that $E$ is \textbf{critical}. Analogous definitions carry over to the map $g$ and the collection $\mathcal{E}_g$.

By the Leau-Fatou flower theorem \cite{Milnor}, there exists a Jordan domain $U_0\subset U$ with smooth boundary except at $0$, called an \textbf{attracting petal} of $0$, such that
\begin{itemize}[leftmargin=1cm]
	\item[(1)] $0\in\partial U_0$ and $f(\overline{U_0})\subset U_0\cup\{0\}$.
	\item[(2)] $f:U_0\rightarrow f(U_0)$ is conformal, and $(\partial U_0\sm \left\{0\right\})\cap P(f)=\emptyset$.
	\item[(3)] $\{f^n|_{U_0}\}$ converges locally and uniformly to $0$ as $n\rightarrow\infty$.
	\item[(4)] For any $z\in U$, there exists an integer $k\geq 1$ such that $f^{k}(z)\in U_0$.
\end{itemize}
Denote by $\langle f\rangle$ the grand orbit of $f$. The quotient $U_0/\langle f\rangle$ is conformally isomorphic to the infinite cylinder $\mathbb{C}/\mathbb Z$, which is called the \textbf{attracting cylinder}. Let $\pi$ be the natural projection from $U_0$ to this cylinder. We say that $U_0$ is \textbf{regular} if the arc $\pi(\partial U_0\setminus\{0\})$ lands at the two punctures of the attracting cylinder. Note that every attracting petal contains a regular attracting petal \cite[Proposition~2.15]{CT}. We choose $U_0$ to be regular throughout this paper.

Let $y$ be a parabolic fixed point of a rational map $R$ with multiplier $e^{2\pi i p/q}$ and multiplicity $kq+1$. A union of the corresponding $kq$ attracting petals is called an \textbf{attracting flower}. \added{
	Since, for a simple parabolic map, every Fatou domain other than the completely invariant parabolic domain $U$ is eventually superattracting, there is only one attracting petal at $0$. Hence the chosen regular attracting petal $U_0$ is an attracting flower for $f$.}

Set $V_0:=h(U_0)$. Since $h$ conjugates $f$ to $g$ on $U$, the domain $V_0$ is a regular attracting petal (and hence an attracting flower) for $g$ at $0$. Note that both $U_0$ and $V_0$ are Jordan domains, and hence by the Carathéodory theorem, $h:U_0\to V_0$ extends to a homeomorphism between their boundaries. By the conjugacy relation and the normalization that $0$ is the parabolic fixed point for both maps, we have $h(0)=0$.

For each $n \ge 1$, let $U_n$ be the component of $f^{-n}(U_0)$ containing $U_0$. \added{Since
$U=\bigcup_{n\geq 0} U_n$, there exists an integer $N\geq 1$ such that $U_N$ contains all critical points of $f$ in $U$.  Moreover, since there are only finitely many critical points and non-singleton periodic Julia components, by increasing $N$ if necessary, we may assume that each puzzle piece contains at most one critical point and at most one non-singleton periodic Julia component.}

 Define $\mathcal{P}_n$ to be the collection of all components of $\widehat{\mathbb{C}} \setminus \overline{U_{N+n}}$, called the \textbf{puzzle pieces of depth $n$}. Setting $Z_n = f^{-n}(0) \cap \partial U_n$, it follows that $Z_n \subset Z_{n+1}$ for $n \ge 0$, and $Z_n = f^{-1}(Z_{n-1}) \cap \partial U_n$ for $n \ge N$. 

Let $E_{p}$ denote the Julia component of $f$ containing the parabolic fixed point $0$. The puzzle pieces for the simple parabolic map $f$ have the following properties. 

\vspace{0.3em}
\noindent(P1) Fix $P_n \in \mathcal{P}_n$ for $n \ge 0$. Then
	\begin{enumerate}
		\item[(a)] $P_n$ is a Jordan domain. 
		\item[(b)] $\partial P_n \cap (\widehat{\mathbb{C}} \setminus U) \subset Z_{N+n}$.
		\item[(c)] For any $P_{n+1} \in \mathcal{P}_{n+1}$, if $P_{n+1} \cap P_n \neq \emptyset$, then $P_{n+1} \subset P_n$ and $\partial P_{n+1} \cap \partial P_n \subset Z_{N+n}$. Conversely, any $z \in Z_{N+n} \cap \partial P_n$ belongs to the closure of a unique puzzle piece \added{$P_{n+1}\in\mathcal P_{n+1}$ contained in $P_n$.}
	\end{enumerate}

\vspace{0.3em}
\noindent(P2) For any $n \ge 0$, $\bigcup_{P_n \in \mathcal{P}_n} \overline{P_n} \supset \widehat{\mathbb{C}} \setminus U$ and $\bigcap_{n \ge 0} \bigcup_{P_n \in \mathcal{P}_n} \overline{P_n} = \widehat{\mathbb{C}} \setminus U$.

\vspace{0.3em}
\noindent(P3) For each $E \in \mathcal{E}_f$ and $n \ge 0$, there is a puzzle piece in $\mathcal{P}_n$ whose closure contains $E$.

\begin{proof}[Verification of (P1)--(P3)]
	\renewcommand{\qedsymbol}{} 
	\added{
		Since $\partial U_0$ is a Jordan curve and
		$(\partial U_0\setminus\{0\})\cap P(f)=\emptyset$,
		every component of
		$\partial U_{N+n}\setminus Z_{N+n}$ is a Jordan arc.
		The closures of two distinct such arcs can meet only at points of
		$Z_{N+n}$, and
		$\partial U_{N+n}\setminus Z_{N+n}\subset U$.
		Hence every puzzle piece $P_n\in\mathcal P_n$ is a Jordan domain, and
		$\partial P_n\cap(\widehat{\mathbb C}\setminus U)
		\subset Z_{N+n}$.
		This proves (P1)(a) and (b).}
		
\vspace{0.2em}
\added{Since $U_{N+n}\subset U_{N+n+1}$, every component of $\widehat{\mathbb C}\setminus U_{N+n+1}$ is contained in a unique component of $\widehat{\mathbb C}\setminus U_{N+n}$, so $P_{n+1}\cap P_n\neq\emptyset$ implies $P_{n+1}\subset P_n$. Moreover, 
\[\partial P_{n+1}\cap\partial P_n\subset\partial U_{N+n+1}\cap\partial U_{N+n}\subset Z_{N+n}.\]
Conversely, for any $z\in Z_{N+n}\cap\partial P_n$, we have
$f^{N+n}(z)=0$. If $z$ does not belong to the closure of any
$P_{n+1}\in\mathcal P_{n+1}$ contained in $P_n$, then
$z\in U_{N+n+1}$ and hence $f^{N+n+1}(z)\in U_0$, contradicting
$f^{N+n+1}(z)=0$. Since $\deg_0 f=1$, $\deg_z f^{N+n+1}=\deg_z f^{N+n}$,
so passing from depth $n$ to depth $n+1$ creates no new boundary
branch at $z$. Therefore, there is a unique
$P_{n+1}\in\mathcal P_{n+1}$ contained in $P_n$ such that
$z\in\partial P_{n+1}$.}

\vspace{0.3em}
\added{For (P2), since $U_{N+n}\subset U$, we have
	$	\bigcup_{P_n\in\mathcal P_n}\overline{P_n}=	\widehat{\mathbb C}\setminus U_{N+n}\supset	\widehat{\mathbb C}\setminus U$.
	Moreover, since $U=\bigcup_{n\ge0}U_{N+n}$,
	\[
	\bigcap_{n\ge0}
	\bigcup_{P_n\in\mathcal P_n}\overline{P_n}
	=
	\bigcap_{n\ge0}
	(\widehat{\mathbb C}\setminus U_{N+n})
	=
	\widehat{\mathbb C}\setminus U.
	\]
	This proves (P2).}

\vspace{0.2em}
\added{For (P3), fix $E\in\mathcal E_f$ and $n\geq 0$.
	Suppose first that $E$ is non-singleton. Since $f$ is a simple
	parabolic map, $E$ is a Jordan curve. As $\infty\in U$, we have
	$
	\operatorname{Int}(\widehat{E})
	\subset \widehat{\mathbb C}\setminus
	\overline{U_{N+n}}.
	$
	Since $\operatorname{Int}(\widehat{E})$ is connected, it is contained
	in a unique puzzle piece $P_n\in\mathcal P_n$. Then
	$
	E\subset \overline{P_n}.
	$
	If $E=\{x\}$ is a singleton, then (P2) gives
	$
	x\in\widehat{\mathbb C}\setminus U
	\subset
	\bigcup_{P_n\in\mathcal P_n}\overline{P_n},
	$
	and hence $x\in\overline{P_n}$ for some $P_n\in\mathcal P_n$.}
\end{proof}

For each $n \ge 1$, set $V_n = h(U_n)$. For each $n \ge 0$, let $\mathcal{Q}_n$ denote the collection of all components of $\widehat{\mathbb{C}} \setminus \overline{V_{N+n}}$. This is the puzzle of depth $n$ for the simple parabolic map $g$. 

Since $U_n$ and $V_n$ are finitely-connected domains bounded by finitely many Jordan curves, the Carathéodory theorem implies that $h:U_n\to V_n$ extends to a homeomorphism between their boundaries. Therefore,

\vspace{0.3em}
\noindent(P4) $h$ is well-defined on $\bigcup_{i\ge 0}f^{-i}(0)$ and satisfies $h(f(x))=g(h(x))$ for every $x\in\bigcup_{i\ge 0}f^{-i}(0)$.

\vspace{0.3em}
\noindent(P5) $h$ induces a bijection $\xi_{n}:\mathcal{P}_{n}\rightarrow\mathcal{Q}_{n}$ such that $\xi_{n}(P_{n}):=\operatorname{Int}(h(\partial P_{n}))$.

\vspace{0.3em}
Since $f$ is a simple parabolic map, for every $n\geq0$, the
boundaries of two distinct puzzle pieces $P_n,P_n'\in\mathcal P_n$
intersect if and only if $E_p=\{0\}$ and there exists a point
$z\in Z_{N+n}$ such that $\deg_z f^{N+n}>1$
and $z\in\partial P_n\cap\partial P_n'$.
In this case, exactly $\deg_z f^{N+n}$ puzzle pieces in
$\mathcal P_n$ have their boundaries pinched at $z$.
Consequently,

\vspace{0.3em}
\noindent(P6) If $E_{p}$ is non-singleton, the puzzle pieces of depth $n$ ($n\ge 0$) have pairwise disjoint closures.

\vspace{0.3em}

For each $E\in\mathcal E_f$ and $n\geq0$, let
$P_n^{(1)}(E),\ldots,P_n^{(m)}(E)$ be all puzzle pieces in
$\mathcal P_n$ whose closures contain $E$.
The discussion above shows that $m>1$ if and only if
$E_p=\{0\}$, $E=\{z\}$ for some $z\in Z_{N+n}$, and
$\deg_E\sigma_f^{N+n}>1$. 
We denote
\[
\widetilde P_n(E):=
\bigcup_{j=1}^{m}\overline{P_n^{(j)}(E)}
\]
and set
$\widetilde{\mathcal P}_n
:=\{\widetilde P_n(E)\mid E\in\mathcal E_f\}$.

\noindent(P7) For each $E\in\mathcal E_f$, the sets $\widetilde P_n(E)$
are nested, and $\bigcap_{n\geq0}\widetilde P_n(E)=\widehat E$.

\added{In fact, if $z\in\bigcap_{n\geq0}\widetilde P_n(E)$, then (P2) gives $z\notin U$. The intersection is connected and contains $E$. Since every boundary arc of $\widetilde P_n(E)$ outside the finite set $f^{-(N+n)}(0)$ lies in $U$, no component of $\widehat{\mathbb C}\setminus U$ different from $\widehat E$ can remain in all the sets $\widetilde P_n(E)$. Hence $\bigcap_{n\geq0}\widetilde P_n(E)=\widehat E$.
}

\vspace{0.3em}
One can similarly define $\widetilde{Q}_{n}(E_{g})$ for each $E_{g}\in\mathcal{E}_{g}$ and the corresponding set $\widetilde{\mathcal{Q}}_{n}$ for $g$. 
Thus $\xi_{n}:\mathcal{P}_{n}\rightarrow\mathcal{Q}_{n}$ induces a bijection $\widetilde{\xi}_{n}:\widetilde{\mathcal{P}}_{n}\rightarrow\widetilde{\mathcal{Q}}_{n}$. Specifically, for any $\widetilde{P}_{n}=\bigcup_{j=1}^{m}\overline{P_{n}^{(j)}}\in\widetilde{\mathcal{P}}_{n}$ with $P_{n}^{(j)}\in\mathcal{P}_{n}$, we define
\[
\widetilde{\xi}_{n}(\widetilde{P}_{n})=\bigcup_{j=1}^{m}\overline{\xi_{n}(P_{n}^{(j)})}.
\]
Note that if $\widetilde{P}_{n+1}\subset\widetilde{P}_{n},$ then $\widetilde{\xi}_{n+1}(\widetilde{P}_{n+1})\subset\widetilde{\xi}_{n}(\widetilde{P}_{n})$.

For $E\in\mathcal{E}_{f}$, set
\[
\xi(E):=\partial\Big(\bigcap_{n\ge 0}\widetilde{\xi}_{n}\big(\widetilde{P}_{n}(E)\big)\Big).
\]
\added{The construction is compatible with the corresponding one for $h^{-1} : V \to U$, and hence $\xi$ is a bijection from $\mathcal{E}_f$ onto $\mathcal{E}_g$. Moreover, the boundary extensions of $h$ satisfy $h \circ f = g \circ h$ on every puzzle boundary. Passing to the nested intersections therefore gives $\xi(\sigma_f(E)) = \sigma_g(\xi(E)).$}

Choose a Jordan curve $\eta\subset U$ sufficiently close to
$\sigma_f(E)$, surrounding $\widehat{\sigma_f(E)}$ and containing
no critical values of $f$. Let $\gamma$ be the component of
$f^{-1}(\eta)$ surrounding $\widehat E$. Then
$f:\gamma\to\eta$ is a covering map. Taking $\eta$ sufficiently
close, we may assume that
$\operatorname{Int}(\gamma)\setminus\widehat E$ contains no
critical points of $f$. Hence, by the Riemann--Hurwitz formula,
$f:\gamma\to\eta$ is a covering of degree $\deg_E\sigma_f$.

Since $h\circ f=g\circ h$, the covering
$g:h(\gamma)\to h(\eta)$ has the same degree. Moreover,
$h(\gamma)$ surrounds $\widehat{\xi(E)}$, and, taking $\eta$
sufficiently close, the region between $h(\gamma)$ and
$\widehat{\xi(E)}$ contains no critical points of $g$.
Thus,
\[
\deg_{\xi(E)}\sigma_g
=
\deg(g:h(\gamma)\to h(\eta))
=
\deg_E\sigma_f.
\]

\section{Local quasiconformal conjugacy around a non-singleton periodic Julia component}
Let $f$ be a simple parabolic map and $E_p$ be the Julia component of $f$ containing the parabolic fixed point $0$.
In this section, we divide the construction of the local quasiconformal conjugacy into two cases according to whether $E_p$ is a singleton or not. 
In both cases, let \(E\) be a non-singleton periodic Julia component of \(f\). \addedtwo{By passing to suitable iterates, which preserves the dynamics of simple parabolic maps, we may assume that \(E\) and \(\xi(E)\) are fixed. For ease of notation, we retain \(f, g\) and \(\mathcal{P}_n, \mathcal{Q}_n\) for these iterates and their corresponding puzzles throughout this section.}

\subsection{The case when $E_p$ is non-singleton}
In this subsection, we first define thickened puzzle pieces $\widehat{P}_0(E)$ and $\widehat{Q}_0(\xi(E))$ containing $P_0(E)$ and $Q_0(\xi(E))$, respectively. On these thickened neighborhoods, we construct a local conjugacy $\phi_E:\widehat{P}_0(E)\to \widehat{Q}_0(\xi(E))$. We then prove that this map is quasiconformal and satisfies the required properties (see Proposition~\ref{prop:local_conjugacy}).

\subsubsection{Thickened puzzles of depth 0}
To obtain the local quasiconformal conjugacy surrounding the parabolic fixed point, we take advantage of the following extension lemma from Cui and Tan.
\begin{lemma}\label{Cui-Tan}{\rm(\cite[Lemma 3.4]{CT})}
Let $(R_1,x)$ and $(R_2,y)$ be parabolic fixed points. Let $\phi: \mathcal{V}(R_1) \to \mathcal{V}(R_2)$ be a $K$-quasiconformal conjugacy between their regular attracting flowers. Then for any $\varepsilon > 0$, there exist a neighborhood $W_0$ of $x$ with $\mathcal{V}(R_1) \subset W_0$ and a $(K+\varepsilon)$-quasiconformal map $\phi_0$ on $W_0 \cup R_1(W_0)$ such that $\phi_0 = \phi$ on $\mathcal{V}(R_1)$ and $\phi_0 \circ R_1 = R_2 \circ \phi_0$ on $W_0$.
\end{lemma}

By Lemma~\ref{Cui-Tan}, there exist a neighborhood $W_0$ of $0$ and a quasiconformal map $\phi_0$ on $W_0\cup f(W_0)$ such that $\phi_0=h$ on $U_0\cap W_0$, which is a Jordan domain, and $\phi_0\circ f=g\circ\phi_0$ on $W_0$. \added{Since $E_p$ is non-singleton in the present subsection, $0\notin \bigcup_{i\geq0} f^i(C(f))$. The analogous statement holds for $g$. Hence, after} shrinking $W_0$ if necessary, and still denoting it by $W_0$, we may assume that $W_0$ is a domain containing $0$ with a smooth boundary, and satisfies
\begin{equation}\label{eq:W0_critical}
	W_0 \cap \bigcup_{i=0}^{N+1} f^i(C(f)) = \emptyset, \quad \partial W_0 \cap P(f) = \emptyset,
\end{equation}
and
\begin{equation}\label{eq:phi0_critical}
	\phi_0(W_0) \cap \bigcup_{i=0}^{N+1} g^i(C(g)) = \emptyset, \quad \partial \phi_0(W_0) \cap P(g) = \emptyset.
\end{equation}

For each $0\leq j\leq N+1$, let  $W_j(0)=\operatorname{Comp}_0f^{-j}(W_0)$. By shrinking $W_0$ again, we can assume that \added{$\bigcup_{j=0}^{N+1}W_j(0)$ is contained in the
	neighborhood where $\phi_0\circ f=g\circ\phi_0$ holds.} As $\phi_0=h$ on $\partial U_0\cap W_0$, it follows that 
\[
\tag{$*$}\label{eq:phi0}
\phi_0=h \quad \text{on} \quad \partial U_j\cap W_j(0)\quad \text{for each $0\leq j\leq N+1$.}
\]

The set $\partial P_0(E)\cap f^{-N}(0)$ is finite and nonempty. We write it as the disjoint union of two subsets: $\{x_1,\ldots,x_l\}$, whose points lie on $\partial P_0(E)\cap\partial P_1(E)$, and $\{y_1,\ldots,y_s\}$, whose points lie on $\partial P_0(E)\setminus\partial P_1(E)$. 

\vspace{0.2em}
Let $W_0^*\Subset W_0$ be a Jordan domain containing $0$ with \added{real-analytic} boundary disjoint from $P(f)$. By \eqref{eq:W0_critical}, the inverse branches of $f^{N+1}$ are conformal on $W_{0}$. We define
\[
W(a)=\mathrm{Comp}_{a}f^{-(N+1)}(W_{0}) \quad\text{and}\quad W^{*}(a)=\mathrm{Comp}_{a}f^{-(N+1)}(W_{0}^{*}),
\]
where $a\in\{x_{1},\dots,x_{l},y_{1},\dots,y_{s}\}$. 
By construction, $W^*(a)\Subset W(a)$. After shrinking $W_0$ further, we may assume that all the neighborhoods $W(a)$ are pairwise disjoint.
For the corresponding points on the $g$-plane, set
$W(h(a))=\operatorname{Comp}_{h(a)}g^{-(N+1)}\bigl(\phi_0(W_0)\bigr)$.

Define the thickened puzzle pieces by
\[
\widehat{P}_0(E)=P_0(E)\cup\bigcup_{i=1}^l W(x_i)\cup\bigcup_{j=1}^s W(y_j),
\]
and
\[
\widehat{Q}_0(\xi(E))=Q_0(\xi(E))\cup\bigcup_{i=1}^l W(h(x_i))\cup\bigcup_{j=1}^s W(h(y_j)),
\]
as shown in Figure~\ref{fig:1-1}.

\begin{figure}[htbp]
	\centering
	\includegraphics[width=9cm]{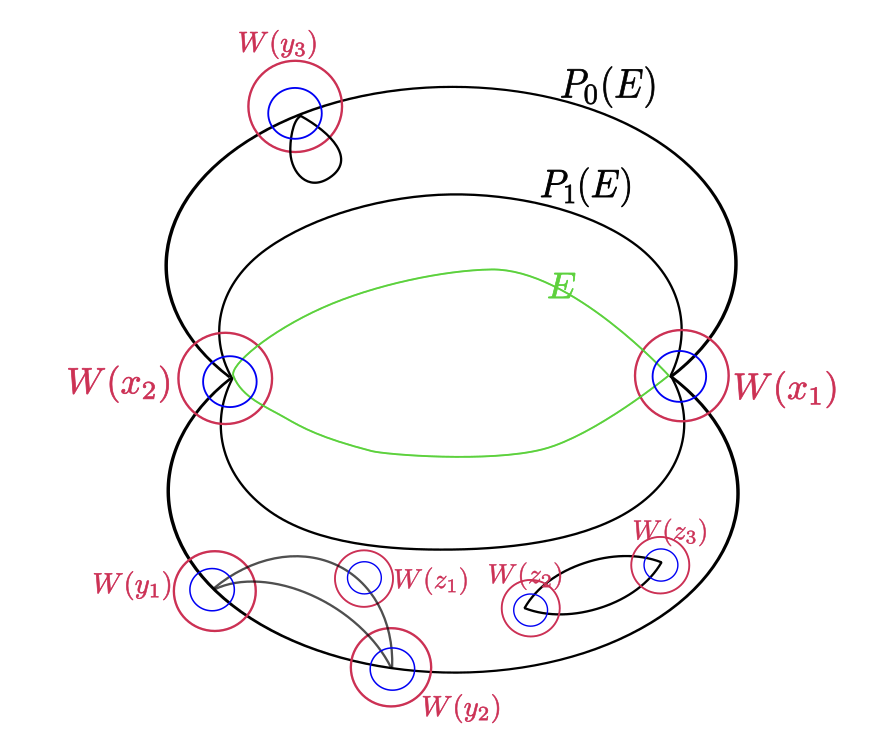}
	\caption{Thickened puzzles of depth 0}
	\label{fig:1-1}
\end{figure}

\subsubsection{The local quasiconformal map $\phi_E$}
For each $n\ge 1$, let
\[
\mathcal B_n(E)=\{P_n\in\mathcal P_n\mid P_n\subset P_{n-1}(E)\ \text{and}\ P_n\neq P_n(E)\}.
\]
\begin{proposition}\label{prop:local_conjugacy}
	There exists a quasiconformal map $\phi_E:\widehat{P}_0(E)\to \widehat{Q}_0(\xi(E))$ such that $\phi_E$ is conformal on
	$\operatorname{Int}(\widehat{E})$, $\phi_E\circ f=g\circ\phi_E$ on $P_1(E)$, and
	\[
	\phi_E=h
	\quad \text{on} \quad
	\widehat{P}_0(E)\setminus
	\big(
	\widehat E\cup
	\bigcup_{n\ge 1}\bigcup_{P_n\in\mathcal B_n(E)} P_n
	\big).
	\]
\end{proposition}

We first construct a homeomorphism $\phi_E:\widehat{P}_0(E)\to \widehat{Q}_0(\xi(E))$ that is quasiconformal on $\widehat{P}_0(E)\setminus E$ and conformal on $\operatorname{Int}(\widehat{E})$. This construction includes three steps.

\medskip
\noindent\textbf{Step 1. Definition on the annulus $\widehat{P}_0(E) \setminus \overline{P_1(E)}$.}
\vspace{0.2em}

The boundaries of puzzle pieces in $\mathcal B_1(E)$ may contain finitely many points in $Z_{N+1}\setminus Z_N$. Denote them by $\{z_1,\dots,z_m\}$. For each such point, define
\[
W(z_k)=\operatorname{Comp}_{z_k}f^{-(N+1)}(W_0),
\qquad
W^*(z_k)=\operatorname{Comp}_{z_k}f^{-(N+1)}(W_0^*),
\]
and
\[
W(h(z_k))=\operatorname{Comp}_{h(z_k)}g^{-(N+1)}\bigl(\phi_0(W_0)\bigr).
\]
Then $W^*(z_k)\Subset W(z_k)$. For each $a \in \{x_1, \dots, x_l, y_1, \dots, y_s, z_1, \dots, z_m\}$, we define a local quasiconformal map $\phi$ on $W(a)$ by
\[
\phi=\bigl(g^{N+1}|_{W(h(a))}\bigr)^{-1}\circ\phi_0\circ f^{N+1}.
\]

\vspace{0.2em}
The annulus $\widehat{P}_0(E)\setminus \overline{P_1(E)}$ admits the decomposition
\[
\widehat{P}_0(E) \setminus \overline{P_1(E)} = \bigg( \widehat{P}_0(E) \setminus \Big( \overline{P_1(E)} \cup \bigcup_{P_1 \in \mathcal{B}_1(E)} \overline{P_1} \Big) \bigg) \cup \bigcup_{P_1 \in \mathcal{B}_1(E)} \overline{P_1}.
\]
On the first part, which is contained in $U$, we set $\phi_E=h$. 

\vspace{0.2em}
Now fix $P_1\in\mathcal B_1(E)$. Its boundary may contain some points among $\{y_1,\dots,y_s\}$ and $\{z_1,\dots,z_m\}$. Figure~\ref{fig:1-2} provides a magnified view of a special case shown in Figure~\ref{fig:1-1}. Set
\[\mathcal{W}^*(P_1)=\bigcup_{y_j \in \partial P_1} W^*(y_j) \cup \bigcup_{z_k \in \partial P_1} W^*(z_k).\]
On $\overline{P_1}\cap \overline{\mathcal W^*(P_1)}$, define $\phi_E=\phi$. By~\eqref{eq:phi0}, we obtain that 
\[
\tag{$**$}\label{eq:phiE}
\phi_E=\phi=h \quad \text{on} \quad \partial U_{N+1}\cap W(a)=\partial P_1\cap W(a) \quad (\text{where } a=y_j \text{ or } z_k).
\]

\begin{figure}[htbp]
	\centering
	\includegraphics[width=8cm]{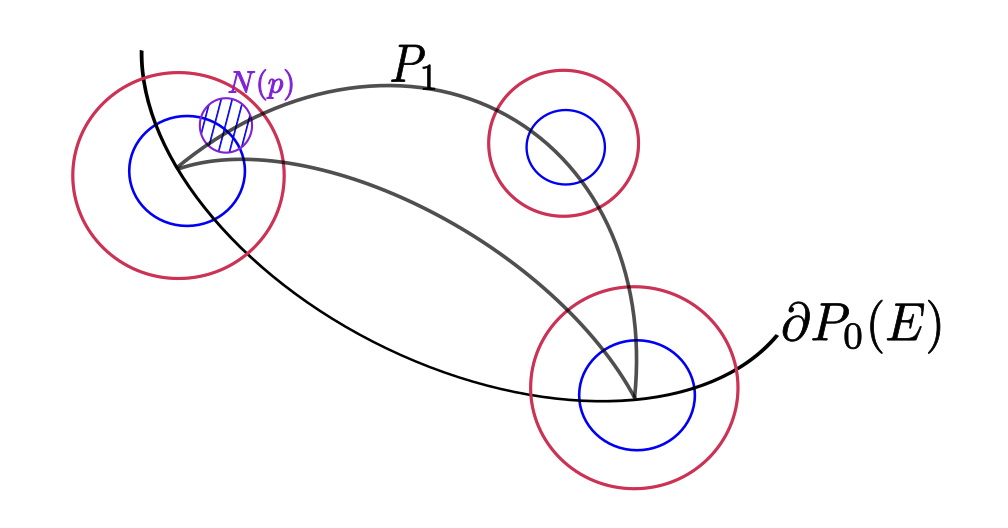}
	\caption{\added{Local geometry near a junction point $p\in\partial P_1\cap\partial W^*(a)$.}}
	\label{fig:1-2}
\end{figure}

It remains to define $\phi_E$ on $\overline{P_1}\setminus \mathcal W^*(P_1)$. \added{By choosing \(W_0^*\) sufficiently small and perturbing its boundary
	slightly if necessary, we may assume that each 
	\(\partial W^*(a)\) meets the puzzle boundaries transversely. Since \(\partial W_0^*\) is real-analytic, so are the curves
	\(\partial W^*(a)\). Thus \(\partial(P_1\setminus \overline{\mathcal W^*(P_1)})\)
	is a piecewise smooth Jordan curve without cusps,
	and hence is a quasicircle.}
	
We define the boundary map $\Psi$ as
\[\Psi(\zeta)=
\begin{cases}
	h(\zeta), & \zeta\in \partial\big(P_1\setminus \overline{\mathcal W^*(P_1)}\big)\cap \partial P_1,\\
	\phi(\zeta), & \zeta\in \partial\big(P_1\setminus \overline{\mathcal W^*(P_1)}\big)\setminus \partial P_1.
\end{cases}
\]
Let $\mathcal C$ be the finite set of \added{junction} points between $\partial P_1$ and $\partial W^*(a)$, where $a=y_j$ or $z_k$. For every $p\in\mathcal C$, since $W^*(a)\Subset W(a)$, there exists a neighborhood $N(p)\Subset W(a)$ (see Figure~\ref{fig:1-2}). Then $\phi=h$ on $N(p)\cap\partial P_1$ by~\eqref{eq:phiE}. Thus, $\Psi$ is locally quasisymmetric at every point of $\mathcal C$. \added{Away from $\mathcal C$, $\Psi$ agrees
	locally with either $h$ or $\phi$, and hence is locally quasisymmetric.
	Since the boundary is compact and $\mathcal C$ is finite, $\Psi$ is
	quasisymmetric on the quasicircle $\partial(P_1\setminus\overline{\mathcal W^*(P_1)})$.} By the Ahlfors--Beurling extension theorem~\cite{Ahlfors}, there exists a quasiconformal map $\phi_E$ on $P_1\setminus\overline{\mathcal W^*(P_1)}$ which extends homeomorphically to its closure and satisfies $\phi_E|_{\partial(P_1\setminus\overline{\mathcal W^*(P_1)})}=\Psi$.

\vspace{0.2em}
Since the two definitions of $\phi_E$ agree on
$\overline{P_1}\cap\partial\mathcal W^*(P_1)$, which consists of finitely many \added{real-analytic arcs, they define a homeomorphism on $P_1$ that is
	quasiconformal off these arcs. By applying the quasiconformal
	removability of real-analytic arcs 
	(cf.~\mbox{\cite[Theorem~8.3]{Lehto}}), we conclude that} $\phi_E$ is quasiconformal on $P_1$, and $\phi_E|_{\partial P_1}=h$ by~\eqref{eq:phiE} and the definition of $\Psi$.
Carrying out this procedure for every $P_1\in\mathcal B_1(E)$ yields a quasiconformal map $\phi_E$ on the annulus $\widehat{P}_0(E)\setminus\overline{P_1(E)}$.

\medskip
\noindent\textbf{Step 2. Inductive lifting.}
\vspace{0.2em}

We extend $\phi_E$ to $\widehat{P}_0(E)\setminus \widehat{E}$ by lifting. More precisely, suppose that $\phi_E$ has been defined on $\widehat{P}_0(E)\setminus \overline{P_n(E)}$. For every puzzle piece $P_{n+1}\in\mathcal B_{n+1}(E)$, $f^n$ conformally maps it onto a puzzle piece in $\mathcal{B}_1(E)$ because $P_0(E) \setminus\widehat{E}$ contains no critical points \added{by the choice of $N$}. Note that $h\circ f^n=g^n\circ h$ on $\partial P_{n+1}$.
We define $\phi_E$ on $P_{n+1}$ by
\[
\phi_E
=
\bigl(g^n|_{Q_{n+1}}\bigr)^{-1}
\circ \phi_E\circ f^n,
\]
where $Q_{n+1}$ is the puzzle piece whose boundary is $h(\partial P_{n+1})$. Then $\phi_E=h$ on $\partial P_{n+1}$.

Repeating this construction for all $n\geq1$, we obtain a quasiconformal map $\phi_E$ on $\widehat{P}_0(E)\setminus\widehat E$ satisfying $\phi_E\circ f=g\circ\phi_E$ on $P_1(E)\setminus\widehat E$. For every $n\geq1$ and every $P_n\in\mathcal B_n(E)$, we have
$\phi_E|_{\partial P_n}=h|_{\partial P_n}$.
More generally, by construction,
\[
\phi_E=h
\quad\text{on}\quad
\widehat{P}_0(E)\setminus
\big(
\widehat E\cup
\bigcup_{n\geq1}\bigcup_{P_n\in\mathcal B_n(E)}P_n
\big).
\]

\noindent\textbf{Step 3. Extension to $\widehat E$.}
\vspace{0.2em}

Since $\widehat{P}_0(E)\setminus \widehat E$ and $\widehat{Q}_0(\xi(E))\setminus \widehat{\xi(E)}$ are doubly-connected, $\phi_E$ can be extended homeomorphically to the boundary component $E$ (cf. \cite{Lehto}). Now we define $\phi_E$ on $\operatorname{Int}(\widehat{E})$. Set $d_0 := \deg_E \sigma_f$.
\addedtwo{By the definition of a simple parabolic map, \(\operatorname{Int}(\widehat{E})\) and \(\operatorname{Int}(\widehat{\xi(E)})\) are fixed superattracting Fatou domains. We denote their unique superattracting fixed points by \(c\) and \(c'\), respectively. By the Riemann-Hurwitz formula, \(c\) and \(c'\) are their unique critical points with local degree \(d_0\).}

	Choose conformal maps
\[
\alpha:\mathbb D\longrightarrow\operatorname{Int}(\widehat{E})
\quad\text{and}\quad
\beta:\mathbb D\longrightarrow\operatorname{Int}(\widehat{\xi(E)})
\]
with $\alpha(0)=c$ and $\beta(0)=c'$, normalized so that
\[
\alpha^{-1}\circ f\circ\alpha(w)=w^{d_0}
\quad\text{and}\quad
\beta^{-1}\circ g\circ\beta(w)=w^{d_0}.
\]
By the Carath\'eodory theorem, $\alpha$ and $\beta$ extend homeomorphically to the closed unit disks. Hence $\phi_E:E\to\xi(E)$ induces an orientation-preserving homeomorphism
\[
\eta:=\beta^{-1}\circ\phi_E\circ\alpha:\mathbb S^1\longrightarrow\mathbb S^1.
\]
Since $\phi_E\circ f=g\circ\phi_E$ on $E$, we have $\eta(w^{d_0})=\eta(w)^{d_0}$ for every $w\in\mathbb S^1$.

\vspace{0.2em}
\added{Set $\kappa=\eta(1)$. Then $\kappa^{d_0-1}=1$. The map $ \kappa^{-1}\eta$ is an orientation-preserving homeomorphism of $\mathbb S^1$ commuting with $w\mapsto w^{d_0}$ and fixing $1$. Moreover, it maps each periodic point of $w\mapsto w^{d_0}$ to a periodic point with the same period.
Since the set of all periodic points of $w\mapsto w^{d_0}$ is dense in $\mathbb S^1$, and the map preserves the cyclic order and fixes $1$, it is the identity. Hence $\eta(w)=\kappa w$.}
This yields $\beta^{-1} \circ \phi_E \circ \alpha(w) = \kappa w$ for $w \in \mathbb{S}^1$. So we can define
$\phi_E(z)=\beta(\kappa\alpha^{-1}(z))$ for all $z\in\operatorname{Int}(\widehat{E})$.

\vspace{0.3em}
Up to now, we have obtained a homeomorphism $\phi_E:\widehat{P}_0(E)\to\widehat{Q}_0(\xi(E))$ satisfying all the properties in Proposition~\ref{prop:local_conjugacy} except for being quasiconformal on $\widehat{P}_0(E)$. 
\vspace{0.2em}

If $\phi$ is a homeomorphism defined in a neighborhood of $x$, we set
\[
\underline{\mathcal H}(\phi,x)
=
\liminf_{r\to0}
\operatorname{Shape}\bigl(\phi(B(x,r)),\phi(x)\bigr)
=
\liminf_{r\to0}
\frac{\sup_{|y-x|=r}|\phi(y)-\phi(x)|}
{\inf_{|y-x|=r}|\phi(y)-\phi(x)|}.
\]

\added{We use the following quasiconformality criterion of Cao, Wang and Yin \cite[Proposition A.1]{CWY}, which is formulated as an analogue of the extremal-length criterion of Kozlovski, Shen and van Strien \cite[Lemma 12.1]{KSS}.}
\begin{proposition}\label{qc criterion}{\rm(\cite[Proposition A.1]{CWY})}
	Let $\phi : \Omega \to \phi(\Omega)$ be an orientation-preserving homeomorphism between domains in $\mathbb{C}$. Let $C \geq 1$ and $M \geq 1$ be constants. Suppose that $\Omega=\bigcup_{k=1}^4\Omega_k$
	and that the following properties hold.
	\begin{enumerate}
		\item[(1)]  $\underline{\mathcal{H}}(\phi,x) \leq C$ for any $x \in \Omega_1$.
		\item[(2)]  $\operatorname{mes}(\phi(\Omega_2)) = 0$ and $\underline{\mathcal{H}}(\phi,x) < \infty$ for any $x \in \Omega_2$.
		\item[(3)]  There is a Markov family $\mathcal{F}$ of Jordan domains in $\mathbb{C}$ such that\footnote{Here the Markov property of $\mathcal{F}$ means that $A_1 \cap A_2 = \emptyset$, or $A_1 \subset A_2$, or $A_2 \subset A_1$ for any $A_1,A_2 \in \mathcal{F}$, and there are at most finitely many $A \in \mathcal{F}$ containing $A_0$ for any $A_0 \in \mathcal{F}$.}
		\begin{itemize}
			\item for any $x \in \Omega_3$, there is a sequence $A_1 \supsetneq A_2 \supsetneq \cdots$ in $\mathcal{F}$ such that $\bigcap_{n=1}^\infty A_n = \{x\}$; and
			\item $\mathcal{F}$ and $\{\phi(A) \mid A \in \mathcal{F}\}$ have $M$-uniform shape.
		\end{itemize}
		\item[(4)]  $\operatorname{length}(\phi(\Omega_4)) = 0$.
	\end{enumerate}
	Then $\phi$ is $K$-quasiconformal, where $K = K(C,M) \geq 1$.
\end{proposition}

\begin{theoremM}{\rm(\cite[Theorem 3.4]{McM88})}
Let $E_R$ be a non-singleton fixed Julia component of a rational map $R$ of degree $d\geq 2$. Then there exist a rational map $\widetilde{R}$ of degree at least $2$ and a quasiconformal map $\varphi: \widehat{\mathbb{C}} \to \widehat{\mathbb{C}}$ such that $\varphi(E_R) = J(\widetilde{R})$ and $\varphi \circ R = \widetilde{R} \circ \varphi$ on $E_R$.
\end{theoremM}

Set $\Omega := \widehat{P}_0(E)$ and decompose it as
\[
\Omega_1 = \widehat{P}_0(E) \setminus E, \quad \Omega_2 = E \setminus \bigcup_{i\ge 0}f^{-i}(0), \quad \Omega_3 = \emptyset, \quad \Omega_4 = E \cap \bigcup_{i\ge 0}f^{-i}(0).
\]

\begin{lemma}\label{lem:zero_area_image_Omega2}
	$\operatorname{mes}(\phi_E(\Omega_2))=0$.
\end{lemma}
\begin{proof}
Since $\phi_E(E)=\xi(E)$, we have $\phi_E(\Omega_2)\subset \xi(E)$. It is enough to prove that $\operatorname{mes}(\xi(E))=0$.
	
Applying McMullen's theorem to the pair $(g,\xi(E))$, we obtain a rational map $R$ and a quasiconformal map $\varphi: \widehat{\mathbb{C}} \to \widehat{\mathbb{C}}$ such that $\varphi(\xi(E))=J(R)$ and $\varphi\circ g=R\circ \varphi$ on $\xi(E)$.

\addedtwo{Since \(g\) is a simple parabolic map, \(\xi(E)\) contains no critical points of \(g\), which makes \(g|_{\xi(E)}\) locally injective. By conjugacy, \(R|_{J(R)}\) is also locally injective. Because $J(R)$ is perfect and completely invariant, this local
	injectivity implies that $J(R)$ contains no critical points of $R$.}

	Hence $R$ is a geometrically finite rational map. A classical result for geometrically finite rational maps states that either $J(R)=\widehat{\mathbb{C}}$ or $\operatorname{mes}(J(R))=0$ (see \cite[Corollary 6.2]{McM00}). The first alternative cannot occur because $J(R)=\varphi(\xi(E))$ and $\xi(E)$ is a Jordan curve. Since quasiconformal maps preserve sets of zero Lebesgue measure, we obtain $\operatorname{mes}(\xi(E))=0$. 
\end{proof}

To show that $\underline{\mathcal H}(\phi_E,x)<\infty$ for every $x\in\Omega_2$, we introduce the bounded degree condition.
Suppose that $R$ is a rational map and $\{\Delta_n(x)\}_{n\ge 0}$ is a sequence of puzzle pieces containing $x$ with $\Delta_{n+1}(x)\subset \Delta_n(x)$. We say that $R$ satisfies the \textbf{bounded degree condition} at $x$ with respect to $\{\Delta_n(x)\}$ if there exist a fixed puzzle piece $\Delta_0$, an infinite sequence $i_n\to\infty$, and a constant $D<\infty$ such that $\deg(R^{i_n}:\Delta_{i_n}(x)\to \Delta_0)\le D$ for all $n\ge 1$.
\vspace{0.2em}


To verify the bounded degree condition on non-singleton Julia components,
we resort to the \textbf{ray-cut puzzle pieces} introduced in the proof of
\cite[Theorem 1.2]{GGP}. \added{Let $E$ be a non-singleton fixed Julia
component, and set 
$d_0=\deg(f|_{\operatorname{Int}(\widehat{E})})$. Let
$\alpha:\mathbb D\to\operatorname{Int}(\widehat{E})$ be a conformal map satisfying
$f(\alpha(w))=\alpha(w^{d_0})$.}

If $0\in E$, let $\gamma_0$ be the internal ray landing at $0$. Otherwise, choose an internal ray landing at a fixed point of $E$ and extend it by an arc in $U$ to $\partial P_0(E)$, as in \cite[Claim~5.1]{GGP}.  Denote the resulting arc by $\gamma_0$. We include the endpoints of $\gamma_0$ and choose it so that $f(\gamma_0)\supset\gamma_0$. Denote its intersection with $E$ by $z_*$.

\added{Fix $0<r<1$ and set $\mathcal{D}_n=\alpha\bigl(\{w:|w|<r^{1/d_0^n}\}\bigr)$.
For $n\geq0$, let $\mathcal X_n(E)$ be the family of components of
\[
P_n(E)\setminus
\bigl(\overline{\mathcal{D}_n}\cup f^{-n}(\gamma_0)\bigr)
\]
that meet $E$. Its elements are Jordan domains, which we call \textbf{ray-cut puzzle pieces}.
If $x\in E\setminus\bigcup_{j\geq0}f^{-j}(z_*)$, let $X_n(x)$ be the unique element of $\mathcal X_n(E)$ containing $x$. If $x\in E\cap f^{-i_0}(z_*)$, then, for $n>i_0$, let
$X_n(x)$ be the Jordan domain obtained by gluing the two adjacent
elements of $\mathcal X_n(E)$ having $x$ on their common boundary
along the common ray-cut boundary arc.  Define $\mathcal Y_n(\xi(E))$ and $Y_n(y)$ for $g$ in the same way.}

\begin{lemma}[Ray-cut puzzle pieces]\label{lem:ray_cut}
\added{For every $x\in E$, we have $X_{n+1}(x)\subset X_n(x)$ and $\bigcap_n\overline{X_n(x)}=\{x\}$, where only depths for which $X_n(x)$ is defined are considered.
	If $x\notin\bigcup_{i\geq0}f^{-i}(0)$, then for every $n\geq1$ there exists $k\geq1$ such that $X_{n+k}(x)\Subset X_n(x)$.
	The same statements hold for $g$ and $\xi(E)$.}
\end{lemma}

\begin{proof}
\added{The nesting property follows from the pullback construction. Since $E$ is a Jordan curve, $\alpha$ extends homeomorphically to
	$\overline{\mathbb D}$. We have
	$\bigcap_n\overline{P_n(E)}=\widehat E$ by~(P7), while
	$\bigcup_n \mathcal{D}_n=\operatorname{Int}(\widehat{E})$. Hence
	$\bigcap_{n}\overline{X_n(x)}\subset E$.
	The set $\overline{X_n(x)}\cap E$ is contained in the image under
	$\alpha$ of one or two adjacent closed arcs of $\mathbb S^1$,
	each of length $2\pi d_0^{-n}$. Their union contains
	$\alpha^{-1}(x)$, so these arcs shrink to $\alpha^{-1}(x)$, proving that
	$\bigcap_{n}\overline{X_n(x)}=\{x\}$.}
	
\added{For $x\notin\bigcup_{i\geq0}f^{-i}(0)$, the set $X_n(x)$ is an open neighborhood of $x$, so  \(X_{n+k}(x) \Subset X_n(x)\) follows from the shrinking of the nested closures. The proof for $g$ is the same.}
\end{proof}


\begin{lemma}\label{lem:bounded_degree_ray_cut}
	For every $x\in E\setminus\bigcup_{i\ge 0}f^{-i}(0)$, $f$ satisfies the bounded degree condition at $x$ with respect to $\{X_n(x)\}_{n\ge 1}$, and analogously for $g$ at any $y\in \xi(E)\setminus\bigcup_{i\ge 0}g^{-i}(0)$ with respect to $\{Y_n(y)\}_{n\ge 1}$.
\end{lemma}

\begin{proof}
	We prove the statement for $f$, and the proof for $g$ is the same.
	\added{Since $E$ contains no critical points, there exists a neighborhood
		$\mathcal N(E)$ of $E$ containing no critical points. By Lemma~\ref{lem:ray_cut}
		and the compactness of $E$, there exists $n_0\geq 1$ such that every
		depth-$n_0$ ray-cut puzzle piece meeting $E$ is contained in
		$\mathcal N(E)$. Hence $X_{n_0}(f^j(x))$ contains no critical points
		for every $j>0$.}  Then $\deg(f|_{X_{n_0+k}(f^j(x))})=1$ for all $k\geq 1$ and $j>0$.
	Hence, by the chain rule for degrees,
	\begin{align*}
		&\deg\left(f^{n_0+n}:X_{n_0+n}(x)\to X_0(f^{n_0+n}(x))\right)\\
		&=\deg(f|_{X_{n_0+n}(x)})\cdot \deg(f|_{X_{n_0+n-1}(f(x))})\cdots \deg(f|_{X_{n_0+1}(f^{n-1}(x))})\cdot \deg(f^{n_0}|_{X_{n_0}(f^n(x))})\\
		&\le \deg(f|_{X_{n_0}(x)})\cdot \deg(f^{n_0}|_{X_{n_0}(f^n(x))})\le D
	\end{align*}
	for a constant $D>0$ independent of $n$. Since the puzzle pieces at a given depth are finite, we can take a subsequence $\{i_n\}$ of $\{n_0+n\}$ such that $X_0(f^{i_n}(x))$ is always the same puzzle piece $X_0$. This proves the bounded degree condition.
\end{proof}

\begin{lemma}\label{lem:parabolic_escape}
Assume that $0\in E$. Let $x\in E\setminus\bigcup_{i\geq0}f^{-i}(0)$ be a point such that $f^{i_n}(x)\to0$ along a sequence $i_n\to\infty$ satisfying the bounded degree condition. Then there exist an increasing sequence of integers $m_n \ge 1$ and a fixed depth-2 puzzle piece $X_2 \Subset X_{0,0}$, where $X_{0,0}$ denotes the depth-0 ray-cut puzzle piece with $0\in\partial X_{0,0}$, such that for $k_n := i_n + m_n - 1$, we have $f^{k_n}(x) \in X_2$ and $\deg(f^{k_n}: X_{k_n}(x) \to X_{0,0})$ is bounded.
\end{lemma}

\begin{proof}
	Let $X_{n,0}$ denote the depth-$n$ ray-cut puzzle piece with $0\in\partial X_{n,0}$. \added{Since $0$ is not a critical point of $f$ and $\bigcap_{n\geq0}\overline{X_{n,0}}=\{0\}$, by taking $X_{0,0}$ sufficiently deep and relabeling the depths, we may assume that $X_{0,0}$ contains no critical point of $f$.} Since $f^{i_n}(x)\to0$, we may assume that $f^{i_n}(x)\in X_{0,0}$ for all sufficiently large $n$. 
	
	Because $\bigcap_{m\geq0}\overline{X_{m,0}}=\{0\}$, there exists an increasing sequence of integers $m_n\geq1$ such that $f^{i_n}(x)\in X_{m_n,0}\setminus\overline{X_{m_n+1,0}}$. \added{By construction, each map $f:X_{j+1,0}\to X_{j,0}$ has degree one, since $X_{j+1,0}\subset X_{0,0}$ contains no critical point of $f$. Hence} $f^{m_n-1}:X_{m_n-1,0}\to X_{0,0}$ is conformal, which implies that the restricted map 
	$$ f^{m_n-1}: X_{m_n+1}(f^{i_n}(x)) \to X_2(f^{i_n+m_n-1}(x)) $$
	is also conformal.
	
	Set $k_n := i_n + m_n - 1$. Note that $X_{k_n}(x)$ is the component of $f^{-k_n}(X_{0,0})$ containing $x$. By Lemma~\ref{lem:bounded_degree_ray_cut}, $\deg(f^{k_n}: X_{k_n}(x) \to X_{0,0}) \le D$ for a constant $D>0$. We have
	\[
	\begin{tikzcd}[column sep=large, row sep=small]
		X_{k_n}(x)
		\arrow[r, "f^{i_n}"]
		& X_{m_n-1,0}
		\arrow[r, "f^{m_n-1}"]
		& X_{0,0} \\
		X_{k_n+2}(x)
		\arrow[u, hook]
		\arrow[r, "f^{i_n}"']
		& X_{m_n+1}(f^{i_n}(x))
		\arrow[u, hook]
		\arrow[r, "f^{m_n-1}"']
		& X_2(f^{k_n}(x))
		\arrow[u, hook]
	\end{tikzcd}
	\]
	Since $\big(\partial X_{1,0} \cap \partial X_{0,0}\big)\cap \overline{U} = \{0\}$ and $\partial X_2(f^{k_n}(x))$ does not contain $0$, we obtain $X_2(f^{k_n}(x)) \Subset X_{0,0}$. Since there are finitely many depth-2 puzzle pieces contained strictly in $X_{0,0}$, passing to a subsequence allows us to assume that $f^{k_n}(x)$ lands in a fixed depth-2 puzzle piece $X_2 \Subset X_{0,0}$.
\end{proof}

We need the following lemma for the required estimate of $\underline{\mathcal H}(\phi_E,x)$ on $\Omega_2$.

\begin{lemma}{\rm(\cite[Theorem 1.4]{Hai})}\label{lem:haissinsky}
For any $m>0$ and $d\geq1$, there exists a constant $C=C(m,d)\ge 1$ such that the following holds. Let $G:U\to V$ be a proper holomorphic map of degree $d$ between bounded simply-connected domains in $\mathbb{C}$. Let $D\subset V$ be a simply-connected domain with $ D\Subset V$ and $\operatorname{mod}(V\setminus\overline D)\ge m$. Fix $z\in D$ and assume that $\operatorname{Shape}(D,z)\le M$. If $D'$ is a simply-connected component of $G^{-1}(D)$ and $w\in D'\cap G^{-1}(z)$, then $\operatorname{Shape}(D',w)\le C M$.
\end{lemma}

\begin{lemma}\label{lem:finite_lower_dilatation_Omega2}
For every $x\in\Omega_2$, we have $\underline{\mathcal{H}}(\phi_E,x)<\infty$.
\end{lemma}
\begin{proof}
The idea of this proof is inspired by \cite[Proposition 5]{Zhai}. Let $x\in\Omega_2=E\setminus\bigcup_{i\ge 0}f^{-i}(0)$. By Lemma~\ref{lem:bounded_degree_ray_cut}, there exists a subsequence $i_n\to\infty$, a fixed ray-cut puzzle piece $X_0$ and a constant $D>0$ such that
$\deg(f^{i_n}:X_{i_n}(x)\to X_0)\le D $.
	
Passing to a subsequence, we may assume that $f^{i_n}(x)\to x_0\in\overline{X_0}$. If $x_0=0$, then by Lemma~\ref{lem:parabolic_escape}, we can replace $i_n$ by a sequence $k_n$ such that $f^{k_n}(x)$ lies in a fixed ray-cut puzzle piece $X_2\Subset X_0$, and the degrees $\deg(f^{k_n}:X_{k_n}(x)\to X_0)$ remain uniformly bounded. If $x_0\ne0$, after passing to a subsequence we can also find a fixed puzzle piece $X_2\Subset X_0$ such that $f^{i_n}(x)\in X_2$. Relabeling the resulting sequence as $\{i_n\}$, we may assume that $f^{i_n}(x)\in X_2\Subset X_0$ for all $n$.

Let $\alpha_n:\mathbb{D}\to X_{i_n}(x)$ and $\beta_n:\mathbb{D}\to X_0$ be conformal maps satisfying $\alpha_n(0)=x$, $\beta_n(0)=f^{i_n}(x)$.
Then $F_n:=\beta_n^{-1}\circ f^{i_n}\circ\alpha_n$ is a proper holomorphic map of degree $d_n\leq D$ satisfying $F_n(0)=0$ (see Figure~\ref{fig:2}).
	
\begin{figure}[htbp]
	\centering
	\includegraphics[width=9.5cm]{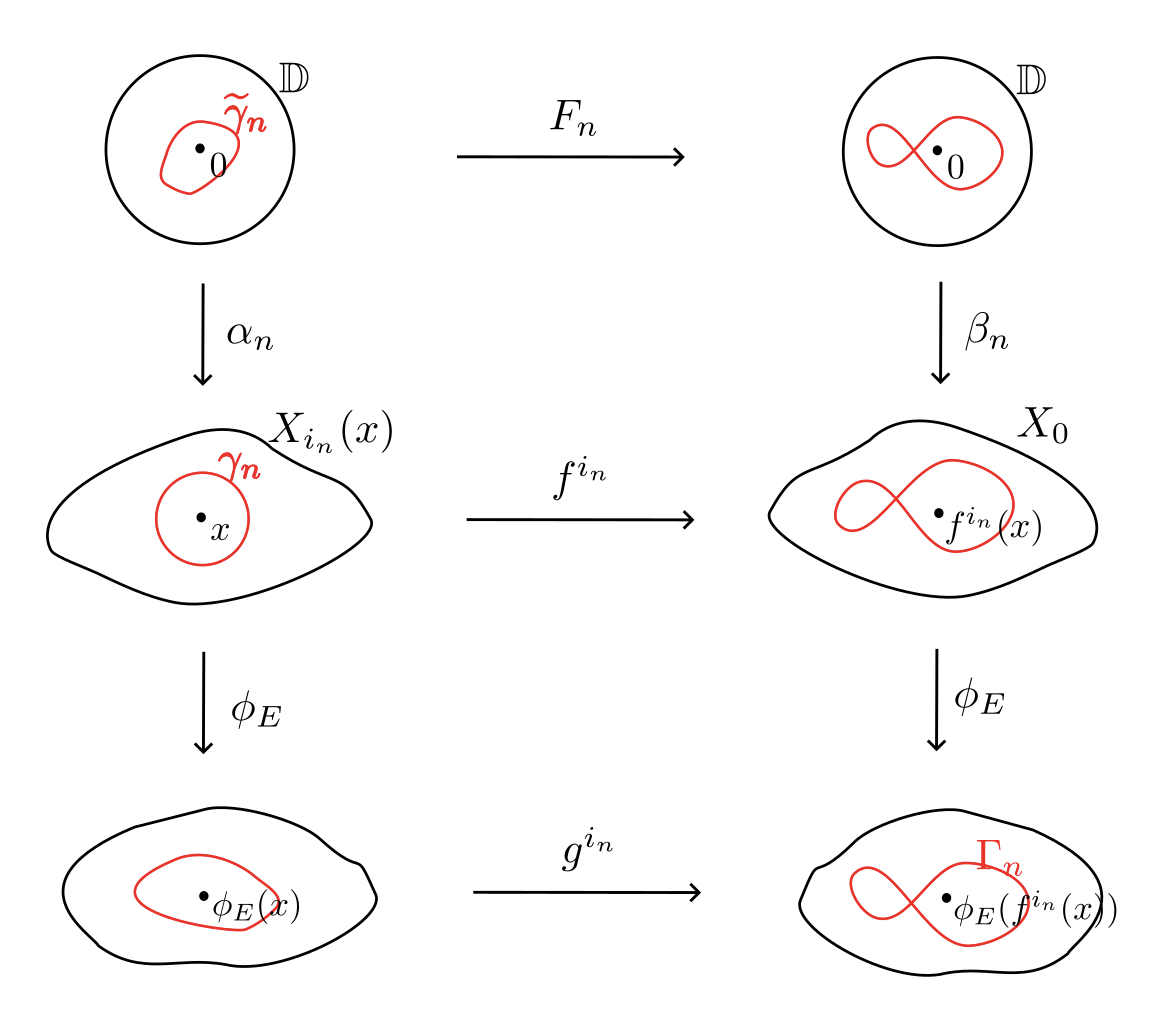}
	\caption{}
	\label{fig:2}
\end{figure}

Let $x=x_1,x_2,\ldots,x_{q_n}$ be the distinct preimages of $f^{i_n}(x)$ in $X_{i_n}(x)$, and let $m_j$ be the multiplicity of $x_j$. Thus $\sum_{j=1}^{q_n}m_j=d_n\le D$. Applying \cite[Lemma 4]{YZ} and its proof to $X_{i_n}(x)$, we obtain a circle $\gamma_n=\partial B(x,r_n)$ such that $r_n<\frac12\operatorname{dist}(x,\partial X_{i_n}(x))$ and
\[
d_{\rho_{X_{i_n}(x)}}(\zeta,x_j)\ge\varepsilon
\quad\text{for all }\zeta\in\gamma_n\text{ and }1\le j\le q_n,
\]
where $d_{\rho_{X_{i_n}(x)}}(\cdot,\cdot)$ denotes the distance under the hyperbolic metric $\rho_{X_{i_n}(x)}$ on $X_{i_n}(x)$, and $\operatorname{dist}(\cdot,\cdot)$ denotes the Euclidean distance. Since $q_n\le D$, the constant $\varepsilon$ can be chosen depending only on $D$.

\vspace{0.2em}
Set $a_{j}:=\alpha_n^{-1}(x_{j})$ and
$\widetilde\gamma_n:=\alpha_n^{-1}(\gamma_n)$. Since $\alpha_n$ is an isometry for the hyperbolic metrics, we have
\[
d_{\rho_{\mathbb{D}}}(z,a_{j})\ge\varepsilon \quad \text{for all} \ z\in\widetilde\gamma_n \ \text{and}\  1\le j\le q_n.
\]
We also have $B(x,2r_n)\subset X_{i_n}(x)$. By the monotonicity of the hyperbolic metric, the hyperbolic diameter of $\gamma_n$ in $ X_{i_n}(x)$ is bounded by the hyperbolic diameter of $\partial B(x,r_n)$ in $B(x,2r_n)$. 
Therefore, there exists a constant $L>0$ such that $\operatorname{diam}_{\rho_{\mathbb{D}}}(\widetilde\gamma_n)\le L$ for all $n$.

\begin{claim}
There exist constants $0<r<R<1$ such that $F_n(\widetilde\gamma_n)\subset\{r<|z|<R\}$ for all $n$.
\end{claim}

\begin{proof}
Since $F_n$ is a finite Blaschke product, it can be written as
\[
F_n(z)=e^{i\theta_n}
\prod_{j=1}^{q_n}
\left(
\frac{z-a_{j}}{1-\overline{a_{j}}z}
\right)^{m_{j}}.
\]
For any $z\in\widetilde{\gamma}_n$, we have 
\[
d_{\rho_{\mathbb{D}}}(\frac{z-a_j}{1-\overline{a_{j}}z},0)=d_{\rho_{\mathbb{D}}}(z,a_j)\geq \varepsilon.
\]
Then
\[
\left|\frac{z-a_{j}}{1-\overline{a_{j}}z}\right|\ge \frac{e^\varepsilon-1}{e^\varepsilon+1}=:\delta.
\]
 It follows that
$|F_n(z)|\ge\delta^{d_n}\ge\delta^D$ on $\widetilde\gamma_n$. Thus the required lower bound holds after taking, for example, $r=\delta^D/2$.

By the Schwarz--Pick lemma \cite{Duren}, $\operatorname{diam}_{\rho_{\mathbb{D}}}F_n(\widetilde\gamma_n)\leq \operatorname{diam}_{\rho_{\mathbb{D}}}\widetilde\gamma_n\leq L$. Moreover, since $\widetilde{\gamma}_n$ surrounds $0$ and $F_n(0)=0$,
$F_n(\widetilde{\gamma}_n)$ has positive winding number about $0$.
Then there exists $R<1$ such that $F_n(\widetilde\gamma_n)\subset \{|z|< R\}$ for all $n$.
\end{proof}

\added{Set
	\[
	\mathcal{D}_n^-:=\phi_E\bigl(\beta_n(B(0,r))\bigr),\qquad \mathcal{D}_n^+:=\phi_E\bigl(\beta_n(B(0,R))\bigr).
	\]
	Let
	\[
	\Gamma_n:=\phi_E\bigl(f^{i_n}(\gamma_n)\bigr)=\phi_E\bigl(\beta_n(F_n(\widetilde{\gamma}_n))\bigr).
	\]
	We have $\Gamma_n\subset \mathcal{D}_n^+\setminus\overline{\mathcal{D}_n^-}$.
	Let $\mathcal{U}_n^-$ and $\mathcal{U}_n^+$ be the components containing $\phi_E(x)$ of $g^{-i_n}(\mathcal{D}_n^-)$ and $g^{-i_n}(\mathcal{D}_n^+)$, respectively. Then \[\mathcal{U}_n^-\subset \phi_E(B(x,r_n))\subset \mathcal{U}_n^+.\]}

\added{Since $\beta_n(0)\in X_2\Subset X_0$ and $0<r<R<1$ are fixed, the compactness of the normalized maps $\beta_n$ and the fact that $\phi_E$ is a homeomorphism give constants $0<a<b$ and $m>0$, independent of $n$, such that
	\[
	B\bigl(\phi_E(f^{i_n}(x)),a\bigr)\subset \mathcal{D}_n^-\subset \mathcal{D}_n^+\subset B\bigl(\phi_E(f^{i_n}(x)),b\bigr)
	\]
	and
	\[
	\operatorname{mod}\bigl(\phi_E(X_0)\setminus\overline{\mathcal{D}_n^+}\bigr)\ge m.
	\]
	In particular, $\operatorname{Shape}(\mathcal{D}_n^+,\phi_E(f^{i_n}(x)))\le b/a$.
	Since $g^{i_n}:\phi_E(X_{i_n}(x))\to\phi_E(X_0)$ is proper of degree at most $D$, Lemma~\ref{lem:haissinsky} gives $\operatorname{Shape}(\mathcal{U}_n^+,\phi_E(x))\le M_1$ for a constant $M_1$ independent of $n$.}

\added{Let $\psi_n:\mathbb D\to  \mathcal{U}_n^+$ be conformal with $\psi_n(0)=\phi_E(x)$. For \(z\in\mathbb D\), the map mapping \(z\) to
	\[
	\frac{g^{i_n}(\psi_n(z))-\phi_E(f^{i_n}(x))}{b}
	\]
	is a holomorphic map from \(\mathbb D\) to itself fixing $0$. 
	By the Schwarz lemma, $\bigl|g^{i_n}(\psi_n(z))-\phi_E(f^{i_n}(x))\bigr|\le b|z|$, and hence $\psi_n(B(0,a/b))\subset \mathcal{U}_n^-$.
	The Koebe $1/4$-theorem therefore gives
	\[
	B\Big(\phi_E(x),\frac{a|\psi_n'(0)|}{4b}\Big)\subset \mathcal{U}_n^-\subset\phi_E(B(x,r_n)).
	\]
	On the other hand, using $\operatorname{dist}(\phi_E(x),\partial \mathcal{U}_n^+)\le|\psi_n'(0)|$ and the shape bound above, we obtain
	\[
	\phi_E(B(x,r_n))\subset \mathcal{U}_n^+\subset B\bigl(\phi_E(x),M_1|\psi_n'(0)|\bigr).
	\]
	Combining the two inclusions, we obtain $\operatorname{Shape}\bigl(\phi_E(B(x,r_n)),\phi_E(x)\bigr)\le\frac{4bM_1}{a}$.}
	
	\vspace{0.2em}
	Since $\bigcap_{k\geq1}\overline{X_k(x)}=\{x\}$, $\operatorname{diam}X_{i_n}(x)\to0$. Thus $r_n\to0$, and hence $\underline{\mathcal{H}}(\phi_E,x)<\infty$.
\end{proof}

\begin{proof}[Proof of Proposition~\ref{prop:local_conjugacy}]
We verify that $\phi_E:\widehat{P}_0(E)\to \widehat{Q}_0(\xi(E))$ satisfies all the required conditions in Proposition~\ref{qc criterion}. The quasiconformality of $\phi_E$ on $\Omega_1$ guarantees (1). The condition~(2) follows directly from Lemmas~\ref{lem:zero_area_image_Omega2} and~\ref{lem:finite_lower_dilatation_Omega2}.
Condition~(3) does not occur because $\Omega_3=\emptyset$. Finally, $\Omega_4$ is countable, and so is $\phi_E(\Omega_4)$. Thus, $\operatorname{length}\bigl(\phi_E(\Omega_4)\bigr)=0$. Therefore, Proposition~\ref{qc criterion} implies that $\phi_E$ is quasiconformal on $\widehat{P}_0(E)$.
\end{proof}

\subsection{The case when $E_p$ is a singleton}
Recall that $E$ is a non-singleton fixed Julia component. Since the forward orbit of $E$ avoids $0$, $E$ is a quasicircle. The component $\xi(E)$ is also a quasicircle. For every $n\geq0$, $P_n(E)$ and $Q_n(\xi(E))$ are neighborhoods of $E$ and $\xi(E)$, respectively.

\begin{proposition}\label{prop:local_conjugacy_singleton}
	There exists a quasiconformal map $\phi_E:P_0(E)\to Q_0(\xi(E))$ such that $\phi_E$ is conformal on $\operatorname{Int}(\widehat{E})$, $\phi_E\circ f=g\circ\phi_E$ on $P_1(E)$, and
                	\[
	\phi_E=h
	\quad\text{on}\quad
	P_0(E)\setminus
	\big(
	\widehat E\cup
	\bigcup_{n\geq1}\bigcup_{P_n\in\mathcal B_n(E)}P_n
	\big).
	\]
\end{proposition}
\begin{proof}
Denote the points of $\big(\partial P_0(E)\setminus \partial P_1(E)\big)\cap f^{-N}(0)$ by $\{y_1, \dots, y_s\}$. The boundaries of the puzzle pieces in $\mathcal{B}_1(E)$ may also contain finitely many points in $Z_{N+1} \setminus Z_N$, which we denote by $\{z_1, \dots, z_m\}$. Thus every $y_j$ and $z_k$ is mapped to $0$ by $f^{N+1}$.

\vspace{0.2em}
As in Subsection 3.1.1, by Lemma~\ref{Cui-Tan}, there exists a neighborhood $W_0$ of $0$ and a quasiconformal map $\phi_0$ on $W_0\cup f(W_0)$ such that $\phi_0=h$ on $U_0\cap W_0$, which is a Jordan domain, and $\phi_0\circ f=g\circ\phi_0$ on $W_0$. 
By shrinking $W_0$ and still denoting it by $W_0$, we may assume that
$W_0$ is a Jordan domain containing $0$ with a smooth boundary,
which satisfies the conditions \eqref{eq:W0_critical}, \eqref{eq:phi0_critical} and \eqref{eq:phi0},
with the equalities to \(\varnothing\) in the first conditions of \eqref{eq:W0_critical} and \eqref{eq:phi0_critical} replaced by inclusions in \(\{0\}\).

For every $a\in\{y_1,\ldots,y_s,z_1,\ldots,z_m\}$, define
\[
W(a)=\operatorname{Comp}_a f^{-(N+1)}(W_0)
\quad \text{and} \quad 
W(h(a))=\operatorname{Comp}_{h(a)}g^{-(N+1)}\bigl(\phi_0(W_0)\bigr).
\]
After shrinking $W_0$ again, we may assume that these neighborhoods are pairwise disjoint, and that $a$ and $h(a)$ are the unique preimages of $0$ in $W(a)$ and $W(h(a))$, respectively. Hence 
\[
f^{N+1} : W(a) \setminus \{a\} \longrightarrow W_0 \setminus \{0\} \quad \text{and} \quad g^{N+1} : W(h(a)) \setminus \{h(a)\} \longrightarrow \phi_0(W_0) \setminus \{0\}
\]
are unbranched coverings.
	
Since $\{a\}$ and all its forward images are singleton Julia components, the degree correspondence for $\xi$ gives $\deg_a f^{N+1}=\deg_{h(a)}g^{N+1}$.
Thus the two coverings have the same degree. We can lift $\phi_0$ through $f^{N+1}$ and $g^{N+1}$ to obtain a quasiconformal map
\[
\phi_a:W(a)\setminus\{a\}
\longrightarrow
W(h(a))\setminus\{h(a)\}.
\]
We choose the lift $\phi_a$ such that $\phi_a = h$ on $\partial P_1 \cap W(a)$ with $P_1\in\mathcal B_1(E)$. This is guaranteed by the condition \eqref{eq:phi0}. Since $h(a)$ is the unique preimage of $0$ in $W(h(a))$, $\phi_a$ extends quasiconformally across $a$ by setting $\phi_a(a)=h(a)$. 

 \added{Using the maps $\phi_a$ for
 	$a\in\{y_1,\ldots,y_s,z_1,\ldots,z_m\}$, we carry out the construction
 	in Steps 1--3 of Subsection~3.1.2. In Step~1, we use $\phi_a$ on the
 	corresponding neighborhoods $W(a)$ and $h$ on the remaining boundary
 	arcs, and extend the resulting boundary maps quasiconformally as before.
 	In Step~2, the map is uniquely extended to the deeper puzzle pieces by
 	lifting through the corresponding conformal branches of $f^n$ and
 	$g^n$. Finally, the map extends homeomorphically to $E$, and we define
 	it conformally on $\operatorname{Int}(\widehat{E})$ as in Step~3, obtaining} a homeomorphism
$\phi_E:P_0(E)\rightarrow Q_0(\xi(E))$ satisfying all the required properties in Proposition~\ref{prop:local_conjugacy_singleton} except for the quasiconformality on $P_0(E)$.

Since $E$ is a quasicircle, it is quasiconformally removable. Note that $\phi_E$ is quasiconformal on $P_0(E)\setminus E$. It follows that $\phi_E$ is quasiconformal on $P_0(E)$.
\end{proof}
	
\section{The global topological conjugacy and quasiconformal rigidity}
In this section, we first extend the local conjugacies constructed in Section~3 to a global topological conjugacy and then prove that this conjugacy is quasiconformal.

\subsection{The global topological conjugacy}
From now on, \(f\) and \(g\) again denote the original simple parabolic maps. We define local quasiconformal conjugacies on all non-singleton Julia components by lifting along the dynamics.

There are finitely many periodic cycles of non-singleton Julia components. For each cycle of period $k$, choose a representative
component $E_0$. By Proposition~\ref{prop:local_conjugacy} or \ref{prop:local_conjugacy_singleton}, there is a map \(\phi_{E_0}:\widehat E_0\to\widehat{\xi(E_0)}\) satisfying $\phi_{E_0} \circ f^k = g^k \circ \phi_{E_0}$ on \(\widehat{E_0}\). 
For a cycle of period $k\geq2$, write
\[
E_{k-1} \longrightarrow E_{k-2} \longrightarrow \cdots \longrightarrow E_1 \longrightarrow E_0,
\]
where \(f(E_j) = E_{j-1}\) for \(1 \le j \le k-1\), and \(f(E_0) = E_{k-1}\).  We successively define \(\phi_{E_1},\ldots,\phi_{E_{k-1}}\) by lifting.

 First we take compatible marked points along the cycle. 
Choose an arc \(\gamma_0\) such that $\gamma_0\cap U$ lands at a fixed point $a_0\in E_0$ of $f^k$
and lies in the region where $\phi_{E_0}=h$.
Set $b_0:=\phi_{E_0}(a_0)$.
Since $\phi_{E_0}$ is continuous at $a_0$, 
$h(\gamma_0\cap U)$ lands at $b_0$. The relation
$\phi_{E_0}\circ f^k=g^k\circ\phi_{E_0}$ gives \(g^k(b_0)=b_0\). For $1\le j\le k-1$, set $a_j:=f^{k-j}(a_0)$ and  $b_j:=g^{k-j}(b_0)$.
Since $f^k(a_0)=a_0$ and $g^k(b_0)=b_0$, we have
$f(a_j)=a_{j-1}$, $g(b_j)=b_{j-1}$ for $1\le j\le k-1$, and $f(a_0)=a_{k-1}$, $g(b_0)=b_{k-1}$.

Suppose that \(\phi_{E_{j-1}}\) has been constructed. If \(E_j\) is critical, let \(c\) and \(\tilde c\) be the unique critical points in \(\widehat E_j\) and \(\widehat{\xi(E_j)}\), respectively. By the definition of a simple parabolic map and the construction of \(\phi_{E_{j-1}}\), we have
\(\phi_{E_{j-1}}(f(c))=g(\tilde c)\). Since
\(\deg_{E_j}\sigma_f=\deg_{\xi(E_j)}\sigma_g\),
\[
f:\widehat E_j\setminus\{c\}\to
\widehat E_{j-1}\setminus\{f(c)\}
\quad\text{and}\quad
g:\widehat{\xi(E_j)}\setminus\{\tilde c\}\to
\widehat{\xi(E_{j-1})}\setminus\{g(\tilde c)\}
\]
are unbranched coverings of the same degree. We lift \(\phi_{E_{j-1}}\) through these coverings, choosing the unique lift that maps \(a_j\) to \(b_j\). Extending this lift quasiconformally across the removable singularity by setting \(\phi_{E_j}(c)=\tilde c\) yields a map
\(\phi_{E_j}:\widehat E_j\to\widehat{\xi(E_j)}\).

If \(E_j\) is non-critical, then \(f\) and \(g\) are injective on
\(\widehat E_j\) and \(\widehat{\xi(E_j)}\), respectively. We define
\(\phi_{E_j}:\widehat E_j\to\widehat{\xi(E_j)}\) as the unique lift of
\(\phi_{E_{j-1}}\) satisfying \(\phi_{E_j}(a_j)=b_j\). In both cases, we obtain that
\[
\phi_{E_{j-1}}\circ f=g\circ\phi_{E_j}
\quad\text{on }\widehat E_j 
\]
for \(j=1,\ldots,k-1\).

It remains to verify that \(\phi_{E_{k-1}}\circ f=g\circ\phi_{E_0}\). By the relations already constructed, we have
\(\phi_{E_0}\circ f^{k-1}=g^{k-1}\circ\phi_{E_{k-1}}\) on \(\widehat E_{k-1}\). Hence, on \(\widehat E_0\),
\[
g^{k-1}\circ(\phi_{E_{k-1}}\circ f)
=\phi_{E_0}\circ f^k
=g^k\circ\phi_{E_0}
=g^{k-1}\circ(g\circ\phi_{E_0}).
\]
Moreover, the two maps \(\phi_{E_{k-1}}\circ f\) and \(g\circ\phi_{E_0}\)  agree at \(a_0\). After removing the possible critical value \(v\) of \(g^{k-1}\), its preimages under \(g^{k-1}\), and the preimages of \(v\) under \(g^k\circ\phi_{E_0}\), the two maps are lifts of the same map through an unbranched covering. Since they agree at \(a_0\), the uniqueness of lifts implies that they coincide off a finite set, and hence on \(\widehat E_0\) by continuity.

For every non-singleton strictly preperiodic Julia component \(E_s\), we successively define \(\phi_{E_s}\) by pulling back along its forward orbit, choosing an iterated preimage of \(0\) as the marked point when available, and otherwise using a pullback of the arc $\gamma_0\cap U$, until \(E_s\) maps to a periodic component.
Therefore, for every non-singleton Julia component \(E_f\), we have 
\begin{equation}\label{eq:commute}
	\phi_{\sigma_f(E_f)} \circ f=g \circ \phi_{E_f} \quad \text{on } \widehat{E_f}.
\end{equation}

We now extend each $\phi_{E_f}$ to a neighborhood of $E_f$. For each chosen representative $E$ of a periodic cycle, set
\[
\mathcal U_E=
\begin{cases}
	\widehat{P}_0(E), & \text{if }E_p\text{ is non-singleton},\\
	P_0(E), & \text{if }E_p=\{0\},
\end{cases}
\qquad
\mathcal V_{\xi(E)}=
\begin{cases}
	\widehat{Q}_0(\xi(E)), & \text{if }E_p\text{ is non-singleton},\\
	Q_0(\xi(E)), & \text{if }E_p=\{0\}.
\end{cases}
\]
The map
$\phi_E:\mathcal U_E\to\mathcal V_{\xi(E)}$
is given by Proposition~\ref{prop:local_conjugacy} or
Proposition~\ref{prop:local_conjugacy_singleton}, respectively.

We first define the thickened puzzle pieces at positive depths when
$E_p$ is non-singleton. Since
$0\notin\bigcup_{i\geq0}f^i(C(f))$ and
$0\notin\bigcup_{i\geq0}g^i(C(g))$, for each $n\geq1$ we may choose
inductively a Jordan domain $\mathcal D_f^n$ containing $0$ such that
\[
\mathcal D_f^1\Subset W_0\cap f(W_0),\qquad
\mathcal D_f^n\Subset
\mathcal D_f^{n-1}\cap f(\mathcal D_f^{n-1})
\quad (n\geq2),
\]
and
\[
\mathcal D_f^n\cap
\bigcup_{i=0}^{N+n}f^i(C(f))=\emptyset,
\qquad
\partial\mathcal D_f^n\cap P(f)=\emptyset.
\]
We choose these domains sufficiently small so that, setting
$\mathcal D_g^n:=\phi_0(\mathcal D_f^n)$, the analogous conditions hold
for $g$. Since $\phi_0\circ f=g\circ\phi_0$ near $0$, we also have
\[
\mathcal D_g^n\Subset
\mathcal D_g^{n-1}\cap g(\mathcal D_g^{n-1})
\quad (n\geq2).
\]

For any $P_n\in\mathcal P_n$, where $n\geq1$, define
$\Lambda(P_n)=\partial P_n\cap f^{-(N+n)}(0)$ and 
\[
\widehat{P}_n
=
P_n\cup
\bigcup_{x\in\Lambda(P_n)}
\operatorname{Comp}_x f^{-(N+n)}(\mathcal D_f^n).
\]
For the corresponding $Q_n\in\mathcal Q_n$, define
\[
\widehat Q_n
=
Q_n\cup
\bigcup_{y\in h(\Lambda(P_n))}
\operatorname{Comp}_y g^{-(N+n)}(\mathcal D_g^n).
\]
By choosing $\mathcal D_f^n$ sufficiently small, the depth-$n$
thickened puzzle pieces are pairwise disjoint and each contains at
most one critical point.

Now fix a non-singleton Julia component $E_f$ which is not a chosen
representative, and let $\ell\geq1$ be the smallest integer such that
$f^\ell(E_f)=E$,
where $E$ is the chosen representative of the periodic cycle in the
forward orbit of $E_f$.

When $E_p$ is non-singleton, choose $n_{E_f}>\ell$ sufficiently large
such that
\[
f^\ell\bigl(\widehat P_{n_{E_f}}(E_f)\bigr)\subset\mathcal U_E,
\qquad
g^\ell\bigl(\widehat Q_{n_{E_f}}(\xi(E_f))\bigr)
\subset\mathcal V_{\xi(E)},
\]
and such that every critical point of $f^\ell$ in
$\widehat P_{n_{E_f}}(E_f)$ lies in $\widehat{E_f}$, while every
critical point of $g^\ell$ in
$\widehat Q_{n_{E_f}}(\xi(E_f))$ lies in
$\widehat{\xi(E_f)}$. Set
\[
\mathcal U_{E_f}:=\widehat P_{n_{E_f}}(E_f),
\qquad
\mathcal V_{\xi(E_f)}:=\widehat Q_{n_{E_f}}(\xi(E_f)).
\]
When $E_p=\{0\}$, choose $n_{E_f}>\ell$ large enough so that
every critical point of $f^\ell$ in $P_{n_{E_f}}(E_f)$ lies in
$\widehat E_f$, while every critical point of $g^\ell$ in
$Q_{n_{E_f}}(\xi(E_f))$ lies in $\widehat{\xi(E_f)}$. Set
\[
\mathcal U_{E_f}:=P_{n_{E_f}}(E_f),
\qquad
\mathcal V_{\xi(E_f)}:=Q_{n_{E_f}}(\xi(E_f)).
\]

In either case,
\[
f^\ell:\mathcal U_{E_f}\setminus\widehat{E_f}
\longrightarrow
f^\ell\bigl(\mathcal U_{E_f}\setminus\widehat{E_f}\bigr)
\quad \text{and} \quad
g^\ell:\mathcal V_{\xi(E_f)}\setminus\widehat{\xi(E_f)}
\longrightarrow
g^\ell\bigl(
\mathcal V_{\xi(E_f)}\setminus\widehat{\xi(E_f)}
\bigr)
\]
are holomorphic coverings of the same degree. We lift $\phi_E$ through these covering maps, choosing the unique lift
whose continuous extension to $E_f$ agrees with the previously
constructed $\phi_{E_f}$. We continue to denote this lift by
$\phi_{E_f}$.
For every $z\in E_f$, $f^\ell$ and $g^\ell$ are conformal in
neighborhoods of $z$ and $\phi_{E_f}(z)$, respectively. By \eqref{eq:commute}, the definitions of  $\phi_{E_f}$ on
$\widehat{E_f}$ and on $\mathcal U_{E_f}\setminus\widehat{E_f}$ are restrictions of the same
local lift near $z$. Therefore, they glue to a quasiconformal map
$\phi_{E_f}:\mathcal U_{E_f}\rightarrow
\mathcal V_{\xi(E_f)}$.

Whenever $\mathcal U_{E_f}$ is chosen at a positive depth, we reindex
the original puzzle pieces inside $\mathcal U_{E_f}$ so that
$\mathcal U_{E_f}$ is regarded as the depth-$0$ puzzle neighborhood
of $E_f$, and define $\mathcal B_n(E_f)$ as in Subsection~3.1.2. Combining Propositions~\ref{prop:local_conjugacy}
and~\ref{prop:local_conjugacy_singleton} with the lifting construction
above, we obtain
\begin{equation}\label{eq:local-agreement}
	\phi_{E_f}=h
	\quad\text{on}\quad
	\mathcal U_{E_f}\setminus
	\Big(
	\widehat{E_f}\cup
	\bigcup_{n\geq1}
	\bigcup_{P_n\in\mathcal B_n(E_f)}P_n
	\Big)
\end{equation}
for every non-singleton Julia component $E_f$.

\begin{proposition}\label{top conj non-singleton}
	Assume that $E_p$ is non-singleton. Then there exists a homeomorphism $H:\widehat{\mathbb{C}}\to\widehat{\mathbb{C}}$ such that $H$ is conformal on the Fatou set $F(f)$, $H\circ f=g\circ H$ on $\widehat{\mathbb{C}}$, and $H=h$ on $U$.
\end{proposition}

\begin{proof}
	For a singleton Julia component $\{x\}\in\mathcal E_f$, we identify $\xi(\{x\})$ with its unique point.
	Define the global map $H \colon \widehat{\mathbb{C}} \to \widehat{\mathbb{C}}$ as
	\[
	H(x) = \begin{cases}
		\phi_{E_f}(x) & \text{if } x \in \widehat{E_f} \text{ for a non-singleton } E_f \in \mathcal{E}_f, \\
		\xi(\{x\}) & \text{if } \{x\} \in \mathcal{E}_f, \\
		h(x) & \text{if } x \in U.
	\end{cases}
	\]
	By \eqref{eq:commute}, together with
	$\xi\circ\sigma_f=\sigma_g\circ\xi$ on singleton Julia components and
	$h\circ f=g\circ h$ on $U$, the map $H$ is a bijection satisfying
	$H\circ f=g\circ H$ on $\widehat{\mathbb C}$. Moreover, $H$ is conformal
	on $F(f)$. So we only need to prove its continuity on the Julia set.
	
	\begin{claim}\label{claim: boundary correspondence}
		For every puzzle piece $P_n\in\mathcal{P}_n$, we have $H|_{\partial P_n}=h|_{\partial P_n}$, and then $H(\partial P_n)=\partial Q_n$.
	\end{claim}
	\begin{proof}
		Let $P_n\in\mathcal{P}_n$ and set $\Lambda(P_n)=\partial P_n\cap f^{-(N+n)}(0)$. Since $\partial P_n\setminus\Lambda(P_n)\subset U$, we immediately have $H=h$ on $\partial P_n\setminus\Lambda(P_n)$. It remains to check the points in $\Lambda(P_n)$ (if any). Let $x\in\Lambda(P_n)$, and suppose that $f^j(x)=y\in E_p$ for a minimal $j\geq 0$. Then $y$ is also an iterated preimage of $0$. By Proposition~\ref{prop:local_conjugacy}, we have $\phi_{E_p}(y)=h(y)$.
		Since $E_p$ is non-singleton, the Julia component $E_x$ containing $x$ is also non-singleton. Consequently, $H(x)=\phi_{E_x}(x)=h(x)$ by the lifting construction of $\phi_{E_x}$. So $H|_{\partial P_n}=h|_{\partial P_n}$. Since $h(\partial P_n)=\partial Q_n$, the claim follows.
	\end{proof}
	
	We now verify that $H(x)\to H(x_0)$ as $x\to x_0$ with $x_0\in J(f)$.
	Suppose that $E=\{x_0\}\in\mathcal{E}_f$ is a singleton component. 
	By the property of puzzle pieces, $\bigcap_{n\ge 0}\overline{P_n(E)}=E$, and $\operatorname{diam}(\overline{P_n(E)})\to 0$ as $n\to\infty$. Analogously for $g$, we have $\bigcap_{n\ge 0}\overline{Q_n(\xi(E))} = \xi(E) = \{H(x_0)\}$. Because $\xi(E)$ is also a singleton, $\operatorname{diam}(\overline{Q_n(\xi(E))})\to 0$ as $n\to\infty$.
	Fix an integer $n\ge 0$. For any $x$ sufficiently close to $x_0$, we have $x\in \overline{P_n(E)}$. By Claim \ref{claim: boundary correspondence}, $H(x)\in \overline{Q_n(\xi(E))}$. It follows that $|H(x)-H(x_0)|\le \operatorname{diam}(\overline{Q_n(\xi(E))})$. Letting $n\to\infty$, we conclude that $H(x)\to H(x_0)$.
	
	Let $x_0$ lie on a non-singleton Julia component $E_f$. By
	\eqref{eq:local-agreement} and the definition of $H$,
	\[
	\phi_{E_f}=H
	\quad\text{on}\quad
	\mathcal U_{E_f}\setminus
	\bigcup_{n\geq1}
	\bigcup_{P_n\in\mathcal B_n(E_f)}P_n.
	\]
	Hence $H(x)\to H(x_0)$ as $x\to x_0$ through this set.
	
	It remains to consider $x\in P_n$ for some
	$P_n\in\mathcal B_n(E_f)$. By the constructions of $H$ and $\phi_{E_f}$, together with
	Claim~\ref{claim: boundary correspondence}, both $H(x)$ and
	$\phi_{E_f}(x)$ lie in $Q_n:=\xi_n(P_n)$. 
	For each non-singleton Julia component $E_f$ and each $n\geq1$, set
	\[\mathcal B_n(\xi(E_f))
	=
	\{Q_n\in\mathcal Q_n
	\mid Q_n\subset Q_{n-1}(\xi(E_f))
	\text{ and }Q_n\neq Q_n(\xi(E_f))\}.\]
	
	\begin{claim}\label{hyperbolic diameter}
		Let $E$ be a non-singleton periodic Julia component, and set
		$A=\widehat Q_0(\xi(E))\setminus\widehat{\xi(E)}$.
		There exists a constant $R>0$ such that for every $n\geq1$
		and $Q_n\in\mathcal B_n(\xi(E))$,
		$\operatorname{diam}_{\rho_A}(Q_n)\leq R$.
	\end{claim}
	
	\begin{proof}
		 By passing to suitable iterates, we may assume that $E$ and $\xi(E)$
		are fixed. Throughout this proof, we retain the notation $f,g$ and their
		corresponding puzzles.
		By (P1), (P6), and (P7), there exists $m\geq2$ such that every
		$Q_m\in\mathcal B_m(\xi(E))$ satisfies
		$Q_m\Subset Q_0(\xi(E))\setminus\widehat{\xi(E)}\subset A$.
		For each such $Q_m$, choose a Jordan domain $\mathcal D(Q_m)$ with
		$Q_m\Subset\mathcal D(Q_m)
		\Subset Q_0(\xi(E))\setminus\widehat{\xi(E)}$.
		
	Fix $n\geq m$ and $Q_n\in\mathcal B_n(\xi(E))$, and set
	$Q_m=g^{n-m}(Q_n)$. Consider
	\[
	g^{n-m}:Q_{n-m}(\xi(E))\longrightarrow Q_0(\xi(E)).
	\]
	All critical values of this map lie in $\widehat{\xi(E)}$.
	Let $\mathcal D_n$ be the component of
	$(g^{n-m})^{-1}(\mathcal D(Q_m))$ containing $Q_n$.
	Since $\mathcal D(Q_m)$ is simply-connected and contains no
	critical values of the above map,
	$g^{n-m}:\mathcal D_n\to\mathcal D(Q_m)$ is conformal.
	Moreover,
	$\mathcal D_n\subset
	Q_{n-m}(\xi(E))\setminus\widehat{\xi(E)}\subset A$.
	Hence its inverse
	$\psi_n:\mathcal D(Q_m)\to A$ is univalent holomorphic and satisfies
	$\psi_n(Q_m)=Q_n$. By the Schwarz--Pick lemma,
	\[
	\operatorname{diam}_{\rho_A}(Q_n)
	\leq
	\operatorname{diam}_{\rho_{\mathcal D(Q_m)}}(Q_m).
	\]
		Since $\mathcal B_m(\xi(E))$ is finite, the right-hand side
		is uniformly bounded. Enlarging the bound to include the
		finitely many pieces of depths $1\leq n<m$ proves the claim.
	\end{proof}
	
	Now let $E_s$ be a non-singleton strictly preperiodic Julia component,
	and choose $\ell\geq1$ such that $f^\ell(E_s)=E$, where $E$ is periodic.
	Set $A_s:=\mathcal V_{\xi(E_s)}\setminus\widehat{\xi(E_s)}$.
	By construction, $g^\ell:A_s\to g^\ell(A_s)$ is a holomorphic
	covering, and $g^\ell(A_s)\cup\widehat{\xi(E)}$ is a neighborhood of
	$\widehat{\xi(E)}$.
	
	For all sufficiently large $n$, every
	$Q_n\in\mathcal B_n(\xi(E_s))$ is mapped conformally by $g^\ell$
	onto one of the sufficiently deep puzzle pieces near $\xi(E)$
	considered in Claim~\ref{hyperbolic diameter}. The inverse
	branches constructed in the proof of Claim~\ref{hyperbolic diameter}  have images contained
	in $g^\ell(A_s)$ for all sufficiently large depths, and hence lift
	through $g^\ell$ to univalent holomorphic maps
	$\widetilde\psi_n:\mathcal D(Q_m)\to A_s$ sending $Q_m$ to $Q_n$.
	By the Schwarz--Pick lemma,
	\[
	\operatorname{diam}_{\rho_{A_s}}(Q_n)
	\leq
	\operatorname{diam}_{\rho_{\mathcal D(Q_m)}}(Q_m).
	\]
	Thus the hyperbolic diameters are uniformly bounded. Enlarging the
	bound includes the finitely many remaining depths.
	
	Consequently, for every non-singleton Julia component $E_f$,
	by the preceding uniform hyperbolic diameter estimates and
	\cite[Lemma~4.9]{CP}, there exists a constant $C>0$,
	depending only on $E_f$, such that
	\[
	\frac{\operatorname{diam}(Q_n)}
	{\operatorname{dist}(Q_n,\xi(E_f))}
	\leq C
	\]
	for every $n\geq1$ and every
	$Q_n\in\mathcal B_n(\xi(E_f))$.
	
	Since  $H(x_0)=\phi_{E_f}(x_0)\in\xi(E_f)$, we have
	$
	\operatorname{dist}(Q_n,\xi(E_f))
	\leq
	|\phi_{E_f}(x)-\phi_{E_f}(x_0)|.
	$
	
	\[
	\begin{aligned}
		|H(x)-H(x_0)|
		&\leq |H(x)-\phi_{E_f}(x)|
		+|\phi_{E_f}(x)-\phi_{E_f}(x_0)|\\
		&\leq \operatorname{diam}(Q_n)
		+|\phi_{E_f}(x)-\phi_{E_f}(x_0)|\\
		&\leq (C+1)|\phi_{E_f}(x)-\phi_{E_f}(x_0)|.
	\end{aligned}
	\]
	As $\phi_{E_f}$ is continuous at $x_0$, it follows that $H$ is continuous at $x_0$.
	
	\added{Thus $H$ is continuous on $\widehat{\mathbb C}$. Since $H$ is a continuous bijection from the compact space $\widehat{\mathbb C}$ onto the Hausdorff space $\widehat{\mathbb C}$, it is a homeomorphism.}
\end{proof}

\begin{proposition}\label{top conj singleton}
	Assume that $E_p=\{0\}$. Then there exists a homeomorphism $H:\widehat{\mathbb{C}}\to\widehat{\mathbb{C}}$ such that $H$ is conformal on the Fatou set $F(f)$, $H\circ f=g\circ H$ on $\widehat{\mathbb{C}}$, and $H=h$ on $U$.
\end{proposition}
\begin{proof}
	Define
	\[
	H(x)=
	\begin{cases}
		\phi_{E_f}(x), & x\in\widehat{E_f}\text{ for a non-singleton }E_f\in\mathcal E_f,\\
		\xi(\{x\}), & \{x\}\in\mathcal E_f,\\
		h(x), & x\in U.
	\end{cases}
	\]
	By \eqref{eq:commute}, together with
	$\xi\circ\sigma_f=\sigma_g\circ\xi$ on singleton Julia components and
	$h\circ f=g\circ h$ on $U$, the map $H$ is a bijection satisfying
	$H\circ f=g\circ H$ on $\widehat{\mathbb C}$. Moreover, $H$ is conformal
	on $F(f)$. Hence it remains to prove the continuity of $H$ on the
	Julia set.

	For every $P_n\in\mathcal P_n$, we have $H|_{\partial P_n}=h|_{\partial P_n}$. Indeed, this is immediate away from $f^{-(N+n)}(0)$, while for every $x\in\partial P_n\cap f^{-(N+n)}(0)$, we have $H(x)=\xi(\{x\})=h(x)$ by definition.
	
	\vspace{0.2em}
	Suppose that $E=\{x_0\}$ with $x_0\notin\bigcup_{i\geq0}f^{-i}(0)$. Then $\{P_n(E) \mid n\geq 0\}$ forms an open neighborhood basis of $x_0$.
	The argument for singleton Julia components in the proof of Proposition~\ref{top conj non-singleton} shows that $H$ is continuous at $x_0$.
	
	Now let $E=\{x_0\}$ with $x_0\in\bigcup_{i\geq0}f^{-i}(0)$. Take $\varepsilon>0$. Choose $n\geq0$ sufficiently large such that $\operatorname{diam}\widetilde{Q}_n(\xi(E))<\varepsilon$. Since $h:U_{N+n}\to V_{N+n}$ extends continuously to the boundary and $H(x_0)=h(x_0)$, we can choose a small neighborhood $\mathcal N$ of $x_0$ such that
	\[\mathcal N\setminus U_{N+n}\subset\widetilde{P}_n(E)\quad \text{and} \quad |h(x)-h(x_0)|<\varepsilon\quad\text{for every }x\in\mathcal N\cap U_{N+n}.\] 
	If $x\in\mathcal N\cap U_{N+n}$, then $H(x)=h(x)$, and hence $|H(x)-H(x_0)|<\varepsilon$. If $x\in\mathcal N\setminus U_{N+n}$, then $x\in\widetilde{P}_n(E)$ and thus $H(x)\in\widetilde{Q}_n(\xi(E))$. Since $H(x_0)\in\xi(E)\subset\widetilde{Q}_n(\xi(E))$, we have \[|H(x)-H(x_0)|\leq\operatorname{diam}\widetilde{Q}_n(\xi(E))<\varepsilon.\] Therefore $H(x)\to H(x_0)$ as $x\to x_0$.

	Finally, let $x_0$ lie on a non-singleton Julia component $E_f$.
	By \eqref{eq:local-agreement} and $H=h$ on every puzzle boundary,
	the argument for non-singleton Julia components in the proof of
	Proposition~\ref{top conj non-singleton} shows that $H$
	is continuous at $x_0$.
	
	\added{Thus $H$ is continuous on $\widehat{\mathbb C}$. Since $H$ is a continuous bijection from the compact space $\widehat{\mathbb C}$ onto the Hausdorff space $\widehat{\mathbb C}$, it is a homeomorphism.}
\end{proof}

\subsection{Quasiconformal rigidity}
In this subsection, we give a partition of the Julia set following the terminology and notation in \cite{YZ} and verify the quasiconformal criterion (Proposition~\ref{qc criterion}), thereby proving Theorem~\ref{qc conj}.

Let
\[\operatorname{Crit}:=
\Big\{c\in\widehat{\mathbb{C}}\setminus
\big(U\cup\bigcup_{i\geq0}f^{-i}(0)\big)
\mid
c\text{ is a critical point of }f \Big\}.
\]
Set $\mathcal X_f:=\widehat{\mathbb{C}}\setminus\big(U\cup\bigcup_{i\geq0}f^{-i}(0)\big)$. For any $x\in\mathcal X_f$ and $n\geq0$, let $P_n(x)$ denote the unique depth-$n$ puzzle piece containing $x$. For $x,y\in\mathcal X_f$, we write $x\to y$ if for every $n\geq0$, there exists an integer $j\geq1$ such that $f^j(x)\in P_n(y)$. Clearly, $x\to y$ and $y\to z$ imply $x\to z$.
Let $\operatorname{Forw}(x)=\{y\in\mathcal X_f\mid x\to y\}$.
This relation naturally induces an equivalence relation on $\operatorname{Crit}$. For $c_1, c_2 \in \operatorname{Crit}$, we define
$$
c_1 \sim c_2 \iff c_1 = c_2 \quad \text{or}\quad (c_1 \to c_2 \text{ and } c_2 \to c_1).
$$
Let $[c]$ denote the equivalence class of $c$. 

For a critical point $c$ with $c\to c$ and any $c_1,c_2\in[c]$, we say that $P_{n+k}(c_1)$ is a \textbf{child} of $P_n(c_2)$ if
$$
f^k(P_{n+k}(c_1)) = P_n(c_2) \quad \text{and} \quad f^{k-1} \colon P_{n+k-1}(f(c_1)) \to P_n(c_2) \text{ is conformal.}
$$
A critical point $c$ is \textbf{persistently recurrent} if for all $n \ge 0$ and all $c' \in [c]$, $P_n(c')$ has finitely many children. Otherwise, $c$ is \textbf{reluctantly recurrent}.
Based on their recurrence properties, we define the following subsets of $\operatorname{Crit}$.
\begin{align*}
	&\operatorname{Crit_n} =  \big\{ c \in \operatorname{Crit} \mid c \not\to c' \text{ for any } c' \in \operatorname{Crit}  \big\}, \\
	&\operatorname{Crit_p} = \big \{ c \in \operatorname{Crit} \mid c \to c \text{ and } c \text{ is persistently recurrent}  \big\},\\
	&\operatorname{Crit_r} = \big \{ c \in \operatorname{Crit} \mid c \to c \text{ and } c \text{ is reluctantly recurrent} \big\}, \\
	&\operatorname{Crit_{en}} = \big\{ c \in \operatorname{Crit} \mid c \not\to c \text{ and } \exists c' \in \operatorname{Crit_n} \text{ such that } c \to c'  \big\}, \\
	&\operatorname{Crit_{ep}} =  \big\{ c \in \operatorname{Crit} \mid c \not\to c \text{ and } \exists c' \in \operatorname{Crit_p} \text{ such that } c \to c'  \big\}, \\
	&\operatorname{Crit_{er}} =  \big\{ c \in \operatorname{Crit} \mid c \not\to c \text{ and } \exists c' \in \operatorname{Crit_r} \text{ such that } c \to c'  \big\}.		
\end{align*}
Then $\operatorname{Crit} = \operatorname{Crit_n} \cup \operatorname{Crit_p} \cup \operatorname{Crit_r} \cup \operatorname{Crit_{en}} \cup \operatorname{Crit_{ep}} \cup \operatorname{Crit_{er}}$.

We divide the persistently recurrent critical points into two subsets. Set
\[
\operatorname{Crit_{p1}}:=\left\{c\in\operatorname{Crit_p}\mid c\text{ is periodic}\right\},
\qquad
\operatorname{Crit_{p2}}:=\operatorname{Crit_p}\setminus\operatorname{Crit_{p1}}.
\]
Let $\operatorname{Crit}(x) = \{ c \in \operatorname{Crit} \mid x \to c \}$. For any subscript index $\upnu\in\{\text{n, p, r, en, ep, er, p1, p2}\}$, we denote
$\operatorname{Crit}_\upnu(x)=\operatorname{Crit}(x)\cap\operatorname{Crit}_\upnu$.

Define $J_0$ to be the union of all singleton Julia components contained in $\bigcup_{i\geq0}f^{-i}(0)$ (possibly empty), and denote by $J_1$ the union of all non-singleton Julia components.
The remaining singleton Julia components are classified into
\begin{align*}
	J_2 &= \big\{ x\in J(f)\setminus(J_0\cup J_1) \mid \operatorname{Crit}(x)=\emptyset \text{ or } \operatorname{Crit_n}(x)\cup\operatorname{Crit_r}(x)\neq\emptyset \big\},\\
	J_3 &= \big\{ x\in J(f)\setminus(J_0\cup J_1) \mid \operatorname{Crit}(x) = \operatorname{Crit_p}(x)\cup\operatorname{Crit_{ep}}(x) \text{ and } \operatorname{Crit_{ep}}(x)\neq\emptyset \big\},\\
	J_4 &= \big\{ x\in J(f)\setminus(J_0\cup J_1) \mid \operatorname{Crit}(x)=\operatorname{Crit_p}(x)\neq\emptyset \big\}.
\end{align*}
We further define
\[J_{41}=\{x\in J_4\mid \operatorname{Crit}(x)=\operatorname{Crit_{p1}}(x)\neq\emptyset\},
\qquad
J_{42}=J_4\setminus J_{41}.\]
By construction, these sets give the following decomposition of the Julia set:
\[
J(f)=J_0\sqcup J_1\sqcup J_2\sqcup J_3\sqcup J_{41}\sqcup J_{42}.
\]
We define analogous subsets of $J(g)$, such that
\[ J(g) = \widetilde J_0 \sqcup\widetilde J_1 \sqcup\widetilde J_2 \sqcup\widetilde J_3 \sqcup\widetilde J_{41} \sqcup\widetilde J_{42}. \]
Since $H$ is a global topological conjugacy, it is easy to verify that
$H(J_i)=\widetilde J_i$ for $i\in\{0,1,2,3,41,42\}$.

\begin{proposition}\label{mes(J_1)=0}
	$\operatorname{mes}(J_1) = 0$.
\end{proposition}
\begin{proof}
If $E_p$ is non-singleton, then Lemma~\ref{lem:zero_area_image_Omega2} implies that every non-singleton periodic Julia component has zero Lebesgue measure. If $E_p=\{0\}$, every non-singleton periodic Julia component is a quasicircle and hence also has zero Lebesgue measure.

Let $E_s$ be a non-singleton strictly preperiodic Julia component. Then there exists an integer $\ell\geq 1$ such that $f^{\ell}(E_s)=E$, where $E$ is periodic. Since $\operatorname{mes}(E)=0$ and non-constant holomorphic maps pull back sets of zero Lebesgue measure to sets of zero Lebesgue measure, we have $\operatorname{mes}(E_s)=0$.  \added{Since every non-singleton Julia component bounds a distinct Fatou component and there are at most countably many Fatou components, the family of non-singleton Julia components is countable. Hence $\operatorname{mes}(J_1)=0$.}
\end{proof}

\begin{lemma}\label{bounded degree}
For each $x \in J_2 \cup J_3 \cup J_{41}$, the map $f$ satisfies the bounded degree condition at $x$ with respect to $\{P_n(x)\}_{n\geq 1}$.
\end{lemma}
\begin{proof}
\added{For $x \in J_2 \cup J_3$, we use the bounded degree argument of \mbox{\cite[Lemma 7]{YZ}}. That proof uses the Markov property of the puzzle, the above recurrence classification of critical points, and finiteness of the critical set; all three hypotheses hold for the present puzzle. Thus it applies without change.}

For $x\in J_{41}$, choose $c_0\in\operatorname{Crit}(x)$. By the definition of $J_{41}$, we know that $c_0$ is periodic. Then $c_0 \in \widehat{E}$ for a non-singleton periodic component $E\in\mathcal{E}_f$. Let the period of $E$ be $p$.
	
Since $x \to c_0$, for any integer $n \ge 1$, there exists a minimal integer $l_n \ge 0$ such that $f^{l_n}(x) \in P_n(E)$. Note that $P_{l_n+n}(x) = \operatorname{Comp}_x(f^{-l_n}(P_n(E)))$. 
\added{By the minimality of $l_n$, the puzzle pieces
	\[
	P_{l_n+n}(x),\ f(P_{l_n+n}(x)),\ldots,
	f^{l_n-1}(P_{l_n+n}(x))
	\]
	are pairwise disjoint. Hence each critical point of $f$ occurs in this
	pullback chain at most once. By the chain rule for degrees,
	\[
	\deg\bigl(f^{l_n}:P_{l_n+n}(x)\to P_n(E)\bigr)
	\leq \prod_{c\in C(f)}\deg_c f
	\leq \delta^b,
	\]}
	where $\delta=\max_{c\in C(f)}\deg_c f$ and
	$b=\# C(f)$.

\vspace{0.2em}
For each $n \ge 1$, let $A_n = P_n(E) \setminus \overline{P_{n+1}(E)}$ (which need not be an annulus). There exists an integer $N_0 \ge 1$ such that for all $0 \le j < p$, $f^j( P_{j+N_0}(E) \setminus \widehat{E} )$ contains no critical points.
	
\begin{claim}\label{claim: conformal iterate}
For every $n \ge N_0$ and each puzzle piece $I \subset A_n$, there exists an iterate of $f$ that maps $I$ conformally into the target union $\bigcup_{j=0}^{p-1} A_{N_0+j}$.
\end{claim} 

\begin{proof}
	If $N_0 \le n < N_0+p$, then $I$ already lies in the target union. 
	
	For $n \ge N_0+p$, it suffices to show that $f^p$ maps any piece $I \subset A_n$ conformally into $A_{n-p}$. 
	For any $0 \le m < p$, we have
	\[
	f^m(I) \subset f^m(P_n(E) \setminus \overline{P_{n+1}(E)}) \subset f^m(P_{N_0+m}(E) \setminus \widehat{E}).
	\]
	By the choice of $N_0$, these sets contain no critical points. Therefore, $f^p|_I$ is conformal and $f^p(I) \subset A_{n-p}$. 
	After iterating this reduction $k$ times, where $k\ge 1$ is uniquely determined by $n-kp\in [N_0,N_0+p-1]$, any $I\subset A_n$ is conformally mapped into $A_{n-kp}$. 
\end{proof}	

Fix an integer $n \ge N_0$. Since $f^{l_n}(P_{l_n+n}(x)) = P_n(E)$, \added{let $s_n\geq0$ be the
	maximal integer such that
	\[
	f^{l_n}(P_{l_n+n+s_n}(x))=P_{n+s_n}(E).
	\]
	Such an $s_n$ is finite; otherwise, taking intersections of the
	nested puzzle pieces would give
	$f^{l_n}(\{x\})=\widehat E$, a contradiction.
By the maximality of
$s_n$,
	\[
	f^{l_n}(P_{l_n+n+s_n+1}(x))
	\neq P_{n+s_n+1}(E),
	\]
	and hence
	$f^{l_n}(P_{l_n+n+s_n+1}(x))\subset A_{n+s_n}$.}

\vspace{0.2em}
Choose $t_n \in p\mathbb{Z}_{\geq 0}$ such that $N_0 \le n+s_n-t_n < N_0+p$. By Claim~\ref{claim: conformal iterate}, $f^{t_n}$ maps each piece in $A_{n+s_n}$ conformally into the target union. Letting $i_n = l_n + n + s_n + 1$, we obtain the following diagram:
\[
\begin{tikzcd}[column sep=large]
	P_{i_n}(x)
	\arrow[r, "f^{l_n}", "\deg \le \delta^b"']
	& f^{l_n}(P_{i_n}(x))
	\arrow[r, "f^{t_n}", "\deg = 1"']
	& f^{l_n+t_n}(P_{i_n}(x))
	\arrow[r, "f^{\,i_n-l_n-t_n}", "\deg \le \delta^{N_0+p}"']
	& f^{i_n}(P_{i_n}(x)).
\end{tikzcd}
\]
Passing to a subsequence, we may assume that $f^{i_n}(x)$ always lands in the same depth-$0$ puzzle piece $P_0$. Consequently, we have $\deg(f^{i_n}:P_{i_n}(x)\to P_0)\leq \delta^{b+N_0+p}$.
\end{proof}

\begin{proposition}\label{mes=0}
$\operatorname{mes}(J_2\cup J_3\cup J_{41})=0$.
\end{proposition}

\begin{proof}
By the Lebesgue density theorem (see \cite[Section~1.7]{EG}), it is sufficient to show that no point in $J_2\cup J_3\cup J_{41}$ is a Lebesgue density point of $J(f)$. 

Let $x\in J_2\cup J_3\cup J_{41}$. By Lemma~\ref{bounded degree}, there exist a fixed puzzle piece $P_0$, a sequence $i_n\to\infty$, and a constant $D>0$ such that $f^{i_n}(x)\in P_0$ and $\deg(f^{i_n}:P_{i_n}(x)\to P_0)\leq D$.
After passing to a subsequence, we may assume that $f^{i_n}(x)\to x_0\in J(f)$. Let $E_{x_0}$ denote the Julia component of $f$ containing $x_0$.

Assume first that $E_p$ is non-singleton. If
$E_{x_0}\cap\bigcup_{i\geq0}f^{-i}(0)=\emptyset$, choose $k_0\geq1$ such that
$P_{k_0}(x_0)\Subset P_0(x_0)$. If
$E_{x_0}\cap\bigcup_{i\geq0}f^{-i}(0)\neq\emptyset$ but
$x_0\notin\bigcup_{i\geq0}f^{-i}(0)$, then $E_{x_0}$ is non-singleton, and we may choose $k_0\geq1$ such that $X_{k_0}(x_0)\Subset X_0(x_0)$.
Suppose that $x_0\in f^{-j}(0)$ for some $j\geq0$. Replacing $i_n$ by
$i_n+j$ and passing to a subsequence, we may assume that
$f^{i_n}(x)\to0$ and $f^{i_n}(x)\in X_{0,0}$, where $X_{0,0}$ is a fixed
depth-$0$ ray-cut puzzle piece with $0\in\partial X_{0,0}$. By
Lemma~\ref{lem:parabolic_escape}, there exist integers $k_n\geq i_n$ and a fixed
ray-cut puzzle piece $X_2\Subset X_{0,0}$ such that $f^{k_n}(x)\in X_2$ and
$\deg(f^{k_n}:X_{k_n}(x)\to X_{0,0})$ is uniformly bounded.

\vspace{0.2em}
Now assume that $E_p=\{0\}$. If
$x_0\notin\bigcup_{i\geq0}f^{-i}(0)$, choose $k_0\geq1$ such that
$P_{k_0}(x_0)\Subset P_0(x_0)$.
Otherwise, suppose that $x_0\in f^{-j}(0)$ for some $j\geq0$. Replacing $i_n$
by $i_n+j$ and passing to a subsequence, we may assume that
$f^{i_n}(x)\to0$ and $f^{i_n}(x)\in P_{0,0}$, where $P_{0,0}$ is a fixed
depth-$0$ puzzle piece with $0\in\partial P_{0,0}$. Let
$\{P_{n,0}\}_{n\geq0}$ be the puzzle pieces with $0\in\partial P_{n,0}$.
Repeating the proof of Lemma~\ref{lem:parabolic_escape}, with $X_{n,0}$
replaced by $P_{n,0}$, we obtain integers $k_n\geq i_n$ and a fixed puzzle
piece $P_2\Subset P_{0,0}$ such that $f^{k_n}(x)\in P_2$ and
$\deg(f^{k_n}:P_{k_n}(x)\to P_{0,0})$ is uniformly bounded.

\vspace{0.2em}
Relabeling the resulting sequence as $\{i_n\}$, in every case we obtain a fixed puzzle piece $\Delta_0$ and a puzzle piece $\Delta_2\Subset\Delta_0$ such that $f^{i_n}(x)\in\Delta_2$ and $\deg(f^{i_n}:\Delta_{i_n}(x)\to\Delta_0)$ is uniformly bounded. 

Take a disk $B(\eta,r)\Subset\Delta_2\setminus J(f)$ and choose $\eta_n\in\Delta_{i_n+2}(x)$ such that $f^{i_n}(\eta_n)=\eta$. Let $\mathcal D_n$ be the component of $f^{-i_n}(B(\eta,r))$ containing $\eta_n$. Then \added{$\mathcal D_n\cap J(f)=\emptyset$ because $B(\eta,r)\cap J(f)=\emptyset$ and $J(f)$ is completely invariant, and hence}
\[ \frac{\operatorname{mes}(\Delta_{i_n+2}(x)\cap J(f))}{\operatorname{mes}(\Delta_{i_n+2}(x))} \leq 1- \frac{\operatorname{mes}(\mathcal D_n)}{\operatorname{mes}(\Delta_{i_n+2}(x))}. \]

\begin{claim}
There exists a constant $C>0$, independent of $n$, such that
	\[ \frac{\operatorname{mes}(\mathcal D_n)}{\operatorname{mes}(\Delta_{i_n+2}(x))} \geq C. \]
\end{claim}

\begin{proof}
Set $r_n=\operatorname{dist}(\eta_n,\partial \mathcal D_n)$, $s_n=\operatorname{dist}(\eta_n,\partial\Delta_{i_n+2}(x))$, and $R_n=\max_{\zeta\in\partial\Delta_{i_n+2}(x)} |\zeta-\eta_n|$. Then $\operatorname{mes}(\mathcal D_n)\geq\pi r_n^2$ and $\operatorname{mes}(\Delta_{i_n+2}(x))\leq\pi R_n^2$. 
By Lemma~\ref{lem:haissinsky}, there exists a constant $M\geq1$ such that $\operatorname{Shape}(\Delta_{i_n+2}(x),\eta_n)\leq M$ for all $n$.
Then $R_n\leq Ms_n$ and hence
\[ \frac{\operatorname{mes}(\mathcal D_n)}{\operatorname{mes}(\Delta_{i_n+2}(x))} \geq \frac{1}{M^2}\left(\frac{r_n}{s_n}\right)^2. \]
It remains to prove that $r_n/s_n$ has a uniform positive lower bound.

Let $\rho_n:\mathbb{D}\to\Delta_{i_n+2}(x)$ be a conformal map satisfying $\rho_n(0)=\eta_n$, and let $\rho_0:\mathbb{D}\to\Delta_2$ be a conformal map satisfying $\rho_0(0)=\eta$. Set $G_n:=\rho_0^{-1}\circ f^{i_n}\circ\rho_n$. Then $G_n:\mathbb{D}\to\mathbb{D}$ is a proper holomorphic map satisfying $G_n(0)=0$.

\vspace{0.2em}
Set $a=\operatorname{dist}\bigl(0,\partial\rho_0^{-1}(B(\eta,r))\bigr)>0$. Choose $\xi_n\in\partial\rho_n^{-1}(\mathcal D_n)$ such that $r_n=|\rho_n(\xi_n)-\rho_n(0)|$. Since $G_n(\xi_n)\in\partial\rho_0^{-1}(B(\eta,r))$, we have $|G_n(\xi_n)|\geq a$. On the other hand, the Schwarz lemma implies $|G_n(\xi_n)|\leq|\xi_n|$. Therefore, $|\xi_n|\geq a$.

\vspace{0.2em}		
By the Koebe distortion theorem,
\[ r_n =|\rho_n(\xi_n)-\rho_n(0)| \geq |\rho_n'(0)|\frac{|\xi_n|}{(1+|\xi_n|)^2} \geq |\rho_n'(0)|\frac{a}{(1+a)^2}. \]
The Koebe $1/4$-theorem gives $s_n\leq|\rho_n'(0)|$. This proves the claim.
\end{proof}
	
It follows that
\begin{equation}\label{measure}
\frac{\operatorname{mes}(\Delta_{i_n+2}(x)\cap J(f))}{\operatorname{mes}(\Delta_{i_n+2}(x))} \leq1-C. 
\end{equation} 
\added{Since $\operatorname{Shape}(\Delta_{i_n+2}(x),\eta_n)\le M$, setting $d_n:=\operatorname{diam}(\Delta_{i_n+2}(x))$, we have 
\[\operatorname{mes}(\Delta_{i_n+2}(x)) \ge \pi s_n^2\ge\frac{\pi R_n^2}{M^2}\ge  \frac{\pi}{4M^2}d_n^2.\]
	Moreover, $\Delta_{i_n+2}(x)$ shrinks to $x$, and hence $d_n\to0$. If $x$ is a Lebesgue density point of $J(f)$, then, since $\Delta_{i_n+2}(x)\subset B(x,d_n)$, the Lebesgue density theorem would imply
	\[
	\frac{\operatorname{mes}(\Delta_{i_n+2}(x)\setminus J(f))}{\operatorname{mes}(\Delta_{i_n+2}(x))} \le \frac{4M^2}{\pi d_n^2} \operatorname{mes}\bigl(B(x,d_n)\setminus J(f)\bigr) \rightarrow 0.
	\]
	Thus $\frac{\operatorname{mes}(\Delta_{i_n+2}(x)\cap J(f))}{\operatorname{mes}(\Delta_{i_n+2}(x))} \rightarrow 1$, contradicting \eqref{measure}. 
}
Therefore, $x$ is not a Lebesgue density point of $J(f)$.
\end{proof}
	
For every non-singleton Julia component $E_f$, the construction of
$\phi_{E_f}$ and the definition of $H$ give
\[
\tag{$*\!*\!*$}\label{star:agreement}
\phi_{E_f}=H
\quad\text{on}\quad
\mathcal U_{E_f}\setminus
\bigcup_{n\geq1}\bigcup_{P_n\in\mathcal B_n(E_f)}P_n.
\]

\begin{proposition}\label{dilatation}
For each $x\in J_1\cup J_2\cup J_3\cup J_{41}$, we have $\underline{\mathcal{H}}(H, x) < \infty$.
\end{proposition}
\begin{proof}
\added{For $x\in J_2\cup J_3\cup J_{41}$, the proof of Proposition~\ref{mes=0} gives a fixed
	pair of puzzle pieces $\Delta_2\Subset\Delta_0$, a sequence
	$i_n\to\infty$, and a constant $D>0$ such that
	\[
	f^{i_n}(x)\in\Delta_2 \quad \text{and} \quad
	\deg\bigl(f^{i_n}:\Delta_{i_n}(x)\rightarrow\Delta_0\bigr)\le D.
	\]
	Lemma~\ref{lem:haissinsky} gives the corresponding uniform shape estimate, and
	Claim~\ref{claim: boundary correspondence}, together with the construction of $H$, gives the same
	boundary correspondence as in Lemma~\ref{lem:finite_lower_dilatation_Omega2}. 
	Thus the proof of Lemma~\ref{lem:finite_lower_dilatation_Omega2} applies with
	$\Delta_{i_n}(x)$, $\Delta_0$, and $\Delta_2$ in place of the
	ray-cut puzzle pieces used there, and with $H$ in place of
	$\phi_E$, yielding $\underline{\mathcal H}(H,x)<\infty$.}
	
For $x\in J_1$, we follow the strategy from \cite[Proposition 4.5]{CP}. Let $E_x$ be the non-singleton Julia component containing $x$. By \eqref{star:agreement}, $\phi_{E_x}=H$ on $E_x$. Hence there exists a constant $M\ge 1$ such that
\[
\overline{\mathcal{H}}(\phi_{E_x},x)
:= \limsup_{r\to 0}
\frac{\sup_{|y-x|=r}|\phi_{E_x}(y)-H(x)|}
{\inf_{|y-x|=r}|\phi_{E_x}(y)-H(x)|}
\le M.
\]

We use the distance $|\phi_{E_x}(y) - H(x)|$ to control $|H(y) - H(x)|$. By \eqref{star:agreement}, it suffices to bound $|H(y)-H(x)|$ for $y\in P_n$ with $P_n\in\mathcal B_n(E_x)$.
For such $y$, both $\phi_{E_x}(y)$ and $H(y)$ lie in $Q_n = H(P_n)$. By \cite[Lemma 4.9]{CP}, for some constant $C>0$ depending on $E_x$,
\[
\frac{\operatorname{diam}(Q_n)}{\operatorname{dist}(Q_n,\xi(E_x))}\le C.
\]
This yields
\[
\frac{|\phi_{E_x}(y) - H(y)|}{|\phi_{E_x}(y) - H(x)|} \le C \quad \text{and} \quad \frac{|\phi_{E_x}(y) - H(y)|}{|H(y) - H(x)|} \le C.
\]	
From the triangle inequality, we have
\[
\frac{1}{C+1}|\phi_{E_x}(y) - H(x)| \le |H(y) - H(x)| \le (C+1)|\phi_{E_x}(y) - H(x)|.
\]
By the definition of $\underline{\mathcal{H}}(H, x)$, we conclude that 
\[\underline{\mathcal{H}}(H, x) \le (C+1)^2 \overline{\mathcal{H}}(\phi_{E_x}, x) \le (C+1)^2 M < \infty. \]
\end{proof}

\added{For any $c_0 \in \operatorname{Crit_{p2}}$, let $c_0 \in \widetilde{K}_n \subset K_n \subset K'_n$ be the KSS nest associated with the persistently recurrent critical class $[c_0]$; see \mbox{\cite{KSS}}. We shall use the following uniform complex bounds.
}
\begin{proposition}\label{KSS nest}
There exist an integer $n_0$ and positive constants $m$ and $M$ such that
for every $c_0\in\operatorname{Crit_{p2}}$, the associated KSS nest $c_0\in\widetilde K_n\subset K_n\subset K_n'$ satisfies 
\[\operatorname{mod}(K_n'\setminus\overline{K_n})\ge m, \qquad  \operatorname{mod}(K_n\setminus\overline{\widetilde K_n})\ge m, \qquad \operatorname{Shape}(K_n,c_0)\le M\] for all $n\ge n_0$.
\end{proposition}
\begin{proof}
		Fix $c_0\in\operatorname{Crit_{p2}}$. Since $c_0$ belongs to a singleton Julia component and $c_0\notin\bigcup_{i\geq0}f^{-i}(0)$, the puzzle pieces $P_n(c_0)$ shrink to $c_0$, and there exists $k_0\geq1$ such that $P_{k_0}(c_0)\Subset P_0(c_0)$.
		Taking $P_{k_0}(c_0)$ as the initial piece of the KSS nest, the argument in \cite[Lemma~6]{YZ} applies without change. In particular, there exists an integer $n_0$ such that for all $n\geq n_0$,
		\[
		\operatorname{mod}(K'_n\setminus\overline{K_n})\geq m,
		\quad
		\operatorname{mod}(K_n\setminus\overline{\widetilde K_n})\geq m
		\]
		for some $m>0$. Then \cite[Proposition~1]{YZ} gives $\operatorname{Shape}(K_n,c_0)\leq M$.
		Since there are only finitely many critical equivalence classes, the constants $n_0$, $m$, and $M$ can be chosen uniformly for all $c_0\in\operatorname{Crit_{p2}}$.
	\end{proof}

\begin{proof}[Proof of Theorem~\ref{qc conj}]
	Let $H$ be the topological conjugacy fixed above. Choose a closed Jordan disk $\mathcal{D}_{\infty}\Subset U$ containing $\infty$, and set $\Omega=\widehat{\mathbb{C}}\setminus \mathcal{D}_{\infty}$. Then $\Omega$ is a Jordan domain in $\mathbb C$ containing $J(f)$. We decompose 
	\[\Omega_1=\Omega\cap F(f),\qquad \Omega_2=J_1\cup J_2\cup J_3\cup J_{41},\qquad \Omega_3=J_{42},\qquad \Omega_4=J_0.\] 
	Since $H$ is conformal on $F(f)$, we have $\underline{\mathcal H}(H,x)=1$ for every $x\in\Omega_1$. For $x\in\Omega_2$, Proposition~\ref{dilatation} gives $\underline{\mathcal H}(H,x)<\infty$. Moreover, since $H(J_i)=\widetilde J_i$ for $i\in\{1,2,3,41\}$, the corresponding versions of Propositions~\ref{mes(J_1)=0} and~\ref{mes=0} for $g$ imply $\operatorname{mes}\bigl(H(\Omega_2)\bigr)=0$. 
	
	For each $x\in\Omega_3$, choose $c_0\in\operatorname{Crit_{p2}}(x)$ and let $c_0\in\widetilde K_n\subset K_n\subset K_n'$ be the corresponding KSS nest from Proposition~\ref{KSS nest}. Let $l_n\geq0$ be the smallest integer such that $f^{l_n}(x)\in\widetilde K_n$, and set $V_n(x)=\operatorname{Comp}_x f^{-l_n}(K_n)$. The domains $V_n(x)$ are strictly nested and shrink to $x$. Set 
	\[\mathcal F=\{V_n(x)\mid x\in\Omega_3,\ n\geq n_0\}.\]
	By the properties of puzzle pieces, we know that $\mathcal F$ is a Markov family. The uniform modulus and shape bounds of the KSS nests, together with Lemma~\ref{lem:haissinsky}, give a constant $M\geq1$ such that \[\operatorname{Shape}(V_n(x),x)\leq M,\qquad \operatorname{Shape}(H(V_n(x)),H(x))\leq M\] for all $x\in\Omega_3$ and all $n\geq n_0$. Finally, $J_0\subset\bigcup_{i\geq0}f^{-i}(0)$ is countable, and so is $H(J_0)=H(\Omega_4)$. Thus, $\operatorname{length}\bigl(H(\Omega_4)\bigr)=0$.

\vspace{0.2em}
	Therefore, all the conditions of the quasiconformal criterion (Proposition~\ref{qc criterion}) are satisfied, implying that $H$ is quasiconformal on $\Omega$. Since $H$ is conformal on $F(f)$, $H$ is quasiconformal on $\widehat{\mathbb C}$. Thus, $f$ and $g$ are quasiconformally conjugate.
\end{proof}

\section{Absence of invariant line fields}
In this section, we prove that a simple parabolic map $f$ carries no invariant line fields on $J(f)$. The proof combines the almost continuity of line fields with the uniform modulus and shape estimates for KSS nests.

\vspace{0.2em}
We say that a measurable Beltrami differential $\mu(x) \frac{d\bar{x}}{dx}$ on $\widehat{\mathbb{C}}$ is a \textbf{line field} if $|\mu| = 1$ on a measurable set $\Lambda$ of positive measure, and $\mu = 0$ elsewhere. Furthermore, we introduce the following terminology:
\begin{itemize}
	\item $\mu$ is \textbf{$f$-invariant} if $f^* \mu = \mu$ almost everywhere on $\widehat{\mathbb{C}}$.
	\item $\mu$ is \textbf{almost continuous} at a point $x$ if
	\[ 
	\forall \varepsilon > 0, \quad \lim_{r \to 0} \frac{ \operatorname{mes}(\{ y \in B(x, r) \mid |\mu(y) - \mu(x)| \ge \varepsilon \})}{ \operatorname{mes}(B(x, r))} = 0.
	\]
\end{itemize}
By \cite[Lemma~7.1]{RYZ23}, every complex measurable function is almost continuous almost everywhere. Applying this result to the Beltrami coefficient $\mu$ of a line field, we see that any line field is almost continuous almost everywhere.

\vspace{0.2em}
Let $\mathscr H(f)$ denote the collection of holomorphic maps
$\psi:\mathcal U\to\mathcal V$ between open subsets of
$\widehat{\mathbb C}$ such that
$f^i\circ\psi=f^j$ on $\mathcal U$ for some $i,j\in\mathbb N$.
We employ the following criterion of Shen \cite{Shen} to rule out invariant line fields.

\begin{lemma}\label{Shen}{\rm(\cite[Proposition 3.2]{Shen})}
Let $f$ be a rational map of degree at least $2$, and let $x\in J(f)$.
If there exist a constant $C>1$, an integer $D\geq2$,
and a sequence of maps $h_n:\mathcal U_n\to\mathcal V_n$
in $\mathscr H(f)$ with the following properties:
\begin{enumerate}
	\item[(1)] $\mathcal U_n, \mathcal V_n$ are Jordan domains, $\operatorname{diam}(\mathcal U_n) \to 0$ and $\operatorname{diam}(\mathcal V_n) \to 0$
	as $n \to \infty$.
	\item[(2)] $h_n$ is a proper holomorphic map of degree between $2$ and $D$.
	\item[(3)] For some $a \in \mathcal U_n$ such that $h'_n(a) = 0$ and for $b = h_n(a)$, we have 
	$$\operatorname{Shape}(\mathcal U_n, a) \le C \quad \text{and} \quad \operatorname{Shape}(\mathcal V_n, b) \le C.$$
	\item[(4)] $\operatorname{dist}(x,\mathcal U_n) \le C \cdot \operatorname{diam}(\mathcal U_n)$ and $\operatorname{dist}(x,\mathcal V_n) \le C \cdot \operatorname{diam}(\mathcal V_n)$.
\end{enumerate}
Then for any $f$-invariant line field $\mu$, either $\mu(x) = 0$ or $\mu$ is not almost continuous at $x$.
\end{lemma} 

\begin{proposition}\label{J_{42}}
	Suppose that \(\mu\) is an \(f\)-invariant line field on the Julia set, and let \(J_w\) be the set of wandering critical points in \(\operatorname{Crit}\). If \(x \in J_{42} \setminus J_w\), then \(\mu(x) = 0\) or \(\mu\) is not almost continuous at \(x\).
\end{proposition}

\begin{proof}
We use the construction in \cite[Proposition~7.2]{CWY}, with some modifications for our setting. For $x\in J_{42}\setminus J_w$, choose $c_0\in\operatorname{Crit_{p2}}(x)$ and let $\widetilde K_n\subset K_n\subset K_n'$ be the KSS nest around $c_0$ for $n\ge n_0$.
Let $m'_n$ be the depth of $K'_n$. 


\added{We first establish a separation property: for every
	$c\in\operatorname{Crit}\setminus\operatorname{Crit}(x)$,}
	\[
	f^k(P_{k+m_n'}(x))\cap P_{m_n'}(c)=\emptyset,
	\qquad k\geq0. \tag{$\star$}\label{eq:avoid-critical}
	\]
\added{In fact, for $k\geq1$, since $x\nrightarrow c$, there exists a depth
	$N_x(c)$ such that $f^k(x)\notin P_{N_x(c)}(c)$ for all $k\geq1$. 
	For $k=0$, we have $x\neq c$. Indeed, if $x=c$ were preperiodic,
	then $c_0$ would be periodic because $c_0\to c_0$, contradicting
	$c_0\in\operatorname{Crit}_{p2}(x)$; if $x=c$ were wandering,
	then $x\in J_w$, again a contradiction. 
	Since $x$ lies in a singleton Julia component, by increasing
	$N_x(c)$ if necessary, we may also assume that
	$P_{N_x(c)}(x)\cap P_{N_x(c)}(c)=\emptyset$.
	Since $\operatorname{Crit}$ is finite, choose a depth $N_*$
	larger than all $N_x(c)$. Increasing $n_0$ if necessary, we may
	assume that $m_n'\geq N_*$ for all $n\geq n_0$. Then \eqref{eq:avoid-critical} holds.}

\added{We shall also use the following consequence. If $c\nrightarrow c_0$,
there exists a depth $N_0(c)$ such that $f^j(c)\notin P_{N_0(c)}(c_0)$ for all $j\geq1$. 
Since $c_0\to c_0$, we have $c\neq c_0$. As $c_0$ lies in a
singleton Julia component, by increasing $N_0(c)$ if necessary,
we may also assume that $c\notin P_{N_0(c)}(c_0)$.
Since $\mathrm{Crit}$ is finite, we may increase $N_*$ so that
$N_*\geq N_0(c)$ for every $c\nrightarrow c_0$.
As $m_n'\geq N_*$ and $K_n\subset K_n'$, we have $K_n\subset P_{N_*}(c_0)$ for all $n\geq n_0$. 
Therefore,}
\[
f^j(c)\in K_n \quad \text{for some }j\geq0
\quad\Longrightarrow\quad
c\to c_0. \tag{$\star\star$}\label{eq:critical converge}
\]
	

Let $l_n \ge 0$ be the minimal integer such that $f^{l_n}(x) \in \widetilde{K}_n$. Define $\widetilde{V}_n(x) = \operatorname{Comp}_x(f^{-l_n}(\widetilde{K}_n))$, 
\[V_n(x) = \operatorname{Comp}_x(f^{-l_n}(K_n)) \quad\text{and}\quad V'_n(x) = \operatorname{Comp}_x(f^{-l_n}(K'_n)).\] 
\added{By the minimality of $l_n$, the first-entry degree estimate gives a constant $D_1$ depending only on $\operatorname{Crit}(x)$, such that} $\deg(f^{l_n}|_{\widetilde{V}_n(x)}) \le D_1$. Assume $n$ is sufficiently large that $V'_n(x)$ contains no critical points of $f$. 

Let $v_n > 0$ be the smallest integer such that $f^{v_n}(\widetilde{V}_n(x))$ contains a critical point, \added{denoted by $c_n$. Since $\operatorname{Crit}$ is finite, after passing
	to a subsequence, we may assume that $c_n=c$ for all $n$. By~\eqref{eq:avoid-critical}, $x \to c$. Since $v_n\leq l_n$ and
	$f^{l_n-v_n}(c)\in\widetilde K_n\subset K_n$, it follows from
	\eqref{eq:critical converge} that $c\to c_0$.}  Set $\widetilde{\Lambda}_n = f^{v_n}(\widetilde{V}_n(x))$, $\Lambda_n = f^{v_n}(V_n(x))$ and $\Lambda'_n = f^{v_n}(V'_n(x))$ (see Figure~\ref{fig:3}).

\begin{figure}[htbp]
	\centering
	\includegraphics[width=13cm]{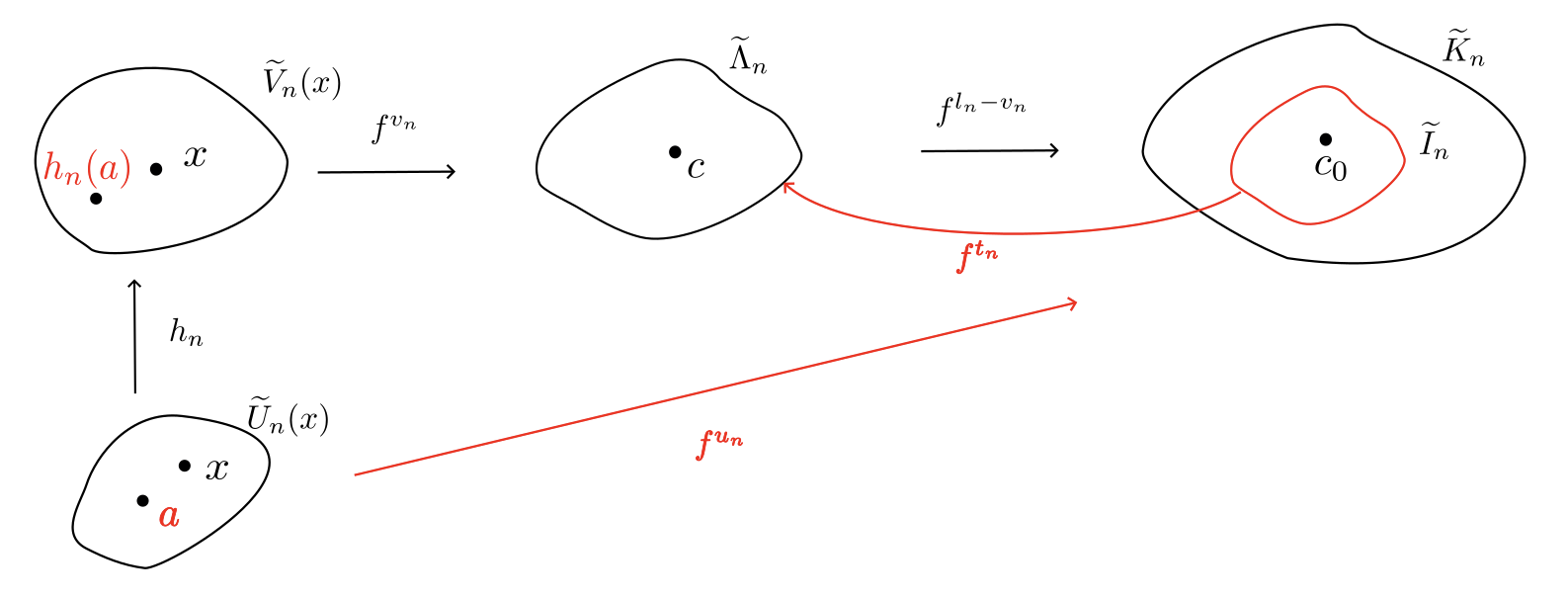}
	\caption{}
	\label{fig:3}
\end{figure}
	
\vspace{0.2em}	
By \cite[Lemma 5.8]{PT} and the fact that $(K_n\setminus\widetilde K_n)\cap\big(\bigcup_{j\geq0}f^j([c_0])\big)=\emptyset$, we obtain
$$ \big(f^i(V_n(x)) \setminus f^i(\widetilde{V}_n(x))\big) \cap \operatorname{Forw}(c_0) = \emptyset$$
for all $n\geq n_0$ and $0\le i\le l_n$.

\begin{claim}
	For all $n \ge n_0$ and $0 \le i \leq l_n$, $\big(f^i(V_n(x)) \setminus f^i(\widetilde{V}_n(x))\big) \cap (\operatorname{Crit} \setminus \operatorname{Forw}(c_0)) = \emptyset$.
\end{claim} 
\begin{proof}
Assume by contradiction that there exist integers $n \ge n_0$, $0 \le i \leq l_n$, and a point $c \in \operatorname{Crit} \setminus \operatorname{Forw}(c_0)$ such that $c \in f^i(V_n(x)) \setminus f^i(\widetilde{V}_n(x))$. 
\added{Since the depth of $K_n'$ is $m_n'$, we obtain 
	\[V_n'(x)=P_{m_n'+l_n}(x)\subset P_{m_n'+i}(x).\]
	Hence $c\in f^i(P_{m_n'+i}(x))$. By~\eqref{eq:avoid-critical}, $x \to c$. Because $x \in J_{42}$, it follows that $c \in \operatorname{Crit_p}(x)$, which implies $\operatorname{Forw}(c) \cap \operatorname{Crit}= [c]$ by \mbox{\cite[Lemma 3]{YZ}}. Note that $f^{l_n-i}(c)\in K_n$. By~\eqref{eq:critical converge},  $c \to c_0$. Then $c_0 \in [c]$. This means $c \in \operatorname{Forw}(c_0)$, which is impossible.}
\end{proof}

Thus $ \big(f^i(V_n(x)) \setminus f^i(\widetilde{V}_n(x))\big) \cap \operatorname{Crit} = \emptyset$. So $\deg(f^{l_n}|_{V_n(x)}) = \deg(f^{l_n}|_{\widetilde{V}_n(x)})$. A similar proof shows $\deg(f^{l_n}|_{V'_n(x)}) = \deg(f^{l_n}|_{V_n(x)})$.
Therefore, 
$$ \deg(f^{l_n}|_{V'_n(x)}) = \deg(f^{l_n}|_{V_n(x)}) = \deg(f^{l_n}|_{\widetilde{V}_n(x)}) \le D_1. $$
By the minimality of $v_n$, $f^{v_n}:V'_n(x)\to\Lambda'_n$ is conformal, which implies
\[
\operatorname{mod}(V_n(x)\setminus\overline{\widetilde V_n(x)})\geq \frac{m}{D_1}
\quad\text{and}\quad
\operatorname{mod}(\Lambda_n\setminus\overline{\widetilde\Lambda_n})\geq \frac{m}{D_1}.
\]
Applying \cite[Lemma 2.2]{RYZ23} and Lemma~\ref{lem:haissinsky}, together with $\operatorname{Shape}(K_n,c_0)\leq M$, we conclude that
\[
\operatorname{Shape}(V_n(x),x)\leq M_1
\quad\text{and}\quad
\operatorname{Shape}(\Lambda_n,c)\leq M_1
\]
for some constant $M_1>0$ independent of $n$.
	
Let \added{$t_n>0$} be the first entry time of $c_0$ into $\widetilde{\Lambda}_n$, and set $\widetilde{I}_n = \text{Comp}_{c_0}(f^{-t_n}(\widetilde{\Lambda}_n))$. Let $u_n$ be the first entry time of $x$ into $\widetilde{I}_n$. We then define $\widetilde{U}_n(x) = \text{Comp}_x f^{-u_n}(\widetilde{I}_n)$,
\[U_n(x) = \text{Comp}_x f^{-(u_n+t_n)}(\Lambda_n) \quad \text{and} \quad U'_n(x) = \text{Comp}_x f^{-(u_n+t_n)}(\Lambda'_n). \]
\added{
	Applying the same argument to
	$U_n'(x)\setminus U_n(x)$ and
	$U_n(x)\setminus \widetilde U_n(x)$, we obtain
	\[
	\deg\bigl(f^{u_n+t_n}|_{U_n'(x)}\bigr)
	=
	\deg\bigl(f^{u_n+t_n}|_{U_n(x)}\bigr)
	=
	\deg\bigl(f^{u_n+t_n}|_{\widetilde U_n(x)}\bigr).
	\]
	By the same first-entry degree estimate, these degrees are bounded by
	a constant $D_2>0$ depending only on $\operatorname{Crit}(x)$.}
Furthermore, it follows that $\operatorname{mod}(U_n(x) \setminus \overline{\widetilde{U}_n(x)}) \ge m/(D_1D_2)$. By \cite[Lemma 2.2]{RYZ23} and Lemma~\ref{lem:haissinsky} again, we obtain $\operatorname{Shape}(U_n(x), x) \le M_2$ for some constant $M_2 > 0$ independent of $n$.
	
\vspace{0.2em}
\added{For each large $n$, define
	\[h_n:=
	\bigl(f^{v_n}|_{V_n(x)}\bigr)^{-1}
	\circ
	\bigl(f^{u_n+t_n}|_{U_n(x)}\bigr).\]
	}Then $h_n:U_n(x)\to V_n(x)$ is a proper holomorphic map of degree between $2$ and $D_2$. 
Let $a\in f^{-u_n}(c_0)\cap\widetilde U_n(x)$. \added{Then
	$f^{u_n}(a)=c_0$. Since $c_0$ is a critical point of $f$,
	we have $f'(c_0)=0$, and hence $\bigl(f^{u_n+t_n}\bigr)'(a)=0$.
	As $\bigl(f^{v_n}|_{V_n(x)}\bigr)^{-1}$ is conformal, it follows that
	$h_n'(a)=0$.}
Set $b=h_n(a)$. Since $f^{v_n}(b)=f^{t_n}(c_0)\in\widetilde\Lambda_n$, we have $b\in\widetilde V_n(x)$. By \cite[Lemma 2.2]{RYZ23},
\[
\operatorname{Shape}(U_n(x),a)\leq C
\quad\text{and}\quad
\operatorname{Shape}(V_n(x),b)\leq C
\]
for some constant $C>0$ independent of $n$. Moreover, since $x\in J_{42}$ lies in a singleton Julia component, \added{the puzzle pieces shrink to $x$, and hence} $\operatorname{diam}(U_n(x))\to0$ and $\operatorname{diam}(V_n(x))\to0$.
Finally, $x\in U_n(x)\cap V_n(x)$, so condition~(4) in Lemma~\ref{Shen} is immediate. Thus all the conditions in Lemma~\ref{Shen} are satisfied, and the proposition holds.
\end{proof}
\begin{proof}[Proof of Theorem~\ref{line field}]
We proceed by contradiction. Suppose that $J(f)$ admits an invariant line field. By definition, there exists a measurable Beltrami differential $\mu(x) \frac{d\bar{x}}{dx}$ supported on a subset $\Lambda \subset J(f)$ of positive Lebesgue measure.

From Propositions \ref{mes(J_1)=0} and \ref{mes=0}, we have $\operatorname{mes}(J_1\cup J_2\cup J_3\cup J_{41})=0$.
Both $J_0$ and $J_w$ are countable, and hence have zero Lebesgue measure.
By the almost continuity of line fields almost everywhere and Proposition~\ref{J_{42}}, we have $\mu(x)=0$ for almost every $x\in J_{42}\setminus J_w$.
It follows that $\mu=0$ almost everywhere on $J(f)$. This leads to $\operatorname{mes}(J(f) \cap \Lambda) = 0$, a contradiction.
\end{proof}

\section{Holomorphic models and perturbations}
In this section, we derive Theorems~\ref{main theorem} and~\ref{perturbation} by the quasiconformal rigidity and the absence of invariant line fields.

\begin{proof}[Proof of Theorem~\ref{main theorem}]
Let $f$ and $g$ be two simple parabolic maps whose completely invariant parabolic Fatou domains are conformally conjugate. By Theorem~\ref{qc conj}, this conformal conjugacy extends to a quasiconformal conjugacy $H:\widehat{\mathbb{C}}\to\widehat{\mathbb{C}}$, which is conformal on $F(f)$.

Assume by contradiction that $H$ is not globally conformal. Let $\mu_H$ denote the complex dilatation of $H$, and set $\Lambda:=\{x\in J(f):|\mu_H(x)|>0\}$. Then $\Lambda$ is contained in $J(f)$ and has positive Lebesgue measure. 
Since $H \circ f = g \circ H$ and both $f$ and $g$ are holomorphic, we have $f^* \mu_H = \mu_H$ almost everywhere. It follows that 
\[
f^*\left(\frac{\mu_H}{|\mu_H|}\right) = \frac{\mu_H}{|\mu_H|} \quad \text{almost everywhere on } \Lambda.
\]

Define a measurable Beltrami differential $\nu(x) \frac{d\bar{x}}{dx}$ as follows:
$$\nu = \begin{cases} 
	\frac{\mu_H}{|\mu_H|} & \text{on } \Lambda, \\ 
	0 & \text{elsewhere}. 
\end{cases}$$
By construction, $|\nu| = 1$ on a set of positive measure and $f^* \nu = \nu$ almost everywhere. Hence $\nu(x) \frac{d\bar{x}}{dx}$ is an invariant line field on $J(f)$, contradicting Theorem~\ref{line field}. Therefore, $H$ is conformal on $\widehat{\mathbb{C}}$. 
\end{proof}

\begin{proof}[Proof of Theorem~\ref{perturbation}]
Let $f$ be a simple parabolic map with completely invariant parabolic Fatou domain $U$.	By \cite[Theorem 1.2]{GGP}, there exists a sequence $\{g_n\}_{n\geq 1}$ of rational maps with completely invariant attracting Fatou domains such that $(f,J(f))$ and $(g_n,J(g_n))$ are topologically conjugate for every $n$. The sequence $\{g_n\}_{n\geq 1}$ converges to a simple parabolic map $g$ with a completely invariant parabolic Fatou domain $V$. Furthermore, $(f,J(f))$ and $(g,J(g))$ are topologically conjugate, while $(f,U)$ and $(g,V)$ are conformally conjugate.

By Theorem~\ref{main theorem}, 
\added{there is a M\"obius transformation $M$ such that
	$M\circ g\circ M^{-1}=f$. Set $f_n:=M\circ g_n\circ M^{-1}$.
	Then $f_n\to f$, and each $f_n$ has a completely invariant attracting Fatou
	domain. Moreover, the topological conjugacy between
	$(f,J(f))$ and $(g_n,J(g_n))$ induces, via $M$, a topological conjugacy
	between $(f,J(f))$ and $(f_n,J(f_n))$. This proves Theorem~\ref{perturbation}.}
\end{proof}

\noindent Ning Gao\\Academy of Mathematics and Systems Science,\\ Chinese Academy of Sciences, Beijing, 100190, P. R. China.\\ gaoning@amss.ac.cn\vskip 0.24cm

\noindent Yan Gao\\School of Mathematical Sciences, \\Shenzhen University, Shenzhen, 518052, P. R. China.\\gyan@szu.edu.cn	\vskip 0.24cm

\noindent Wenjuan Peng \\State Key Laboratory of Mathematical Sciences, Academy of Mathematics and Systems Science,\\Chinese Academy of Sciences, Beijing, 100190, P. R. China.\\wenjpeng@amss.ac.cn
\end{document}